\documentclass{siamart251216}
\newsiamthm{assumption}{Assumption}
\usepackage{lipsum}
\usepackage{amsfonts}
\usepackage{graphicx}
\usepackage{epstopdf}
\usepackage{algpseudocode}
\usepackage{caption}
\usepackage{subcaption}
\usepackage{tikz-cd}
\usepackage{pgfplots}
\usepackage{booktabs}
\usepackage{pgfplotstable}
\usepackage[legacy]{csvsimple}
\usepackage{wrapfig}
\pgfplotsset{compat=1.18}
\usepackage{titlesec}
\usepgfplotslibrary{statistics}
\pgfkeys{/pgf/number format/.cd, fixed, precision=3}

\usepackage{caption}
\usepackage{enumitem}
\setlist{noitemsep,topsep=1pt}

\ifpdf%
  \DeclareGraphicsExtensions{.eps,.pdf,.png,.jpg}
\else
  \DeclareGraphicsExtensions{.eps}
\fi
\usepackage{amsopn}

\usepackage{booktabs}

\def\XVec#1{{\mathbf #1}}       % vector notation

\def\Xc{\XVec{c}}

\def\Xu{\XVec{u}}
\def\Xv{\XVec{v}}

\def\Xf{\XVec{f}}
\def\Xe{\XVec{e}}
\def\Xr{\XVec{r}}
\def\Xn{\XVec{n}}

\def\Xb{\XVec{b}}

\def\Xx{\XVec{x}}

\def\Xz{\XVec{z}}
\def\Xdof{\mathbf{dof}}

\def\XfG{\mathbf{f_G}}
\def\XcG{\mathbf{c_G}}

\newcommand{\ZQS}{{Q_S}}
\newcommand{\ZR}{{R}}
\newcommand{\ZQR}{{Q_R}}
\newcommand{\ZS}{{S}}

\newcommand{\XMt}{{\widetilde{M}}}

\def\XMt{{\widetilde{M}}}
            
\newcommand{\Hcurl}{{H(\mbox{curl})} }            
\newcommand{\tablepanel}[2]{\textbf{(#1)}\;\textit{#2}\par\smallskip}

\newcommand{\TheTitle}{%
Nodal Coarsening and Approximate Ideal Interpolation for $\Hcurl$ in Algebraic Multigrid}

\author{
Taoli Shen\thanks{Department of Mathematics, 
The Pennsylvania State University, University Park, PA 16802, USA 
(\email{tzs5765@psu.edu})}
\and
James Brannick\thanks{Department of Mathematics, 
The Pennsylvania State University, University Park, PA 16802, USA 
(\email{jjb23@psu.edu})}
\and
Robert Falgout\thanks{Lawrence Livermore National Laboratory, 
Livermore, CA 94550, USA 
(\email{falgout2@llnl.gov})}
\and
Karsten Kahl\thanks{Fachgruppe Mathematik \& Informatik, 
Bergische Universität Wuppertal, 42119 Wuppertal, Germany 
(\email{kkahl@uni-wuppertal.de})}
\and
Jacob Schroder\thanks{Department of Mathematics \& Statistics, 
University of New Mexico, Albuquerque, NM 87106, USA 
(\email{jbschroder@unm.edu})}
}
\title{{\TheTitle}\thanks{Submitted to the editors DATE.}}

\headers
  {Approximate Ideal Interpolation for $H(\mathrm{curl})$}
  {T. Shen et al.}

\begin{document}
% Avoid stretched gaps around displays and allow tables to fit nearby pages.
\raggedbottom
\renewcommand{\topfraction}{0.95}
\renewcommand{\bottomfraction}{0.85}
\renewcommand{\textfraction}{0.05}
\renewcommand{\floatpagefraction}{0.75}

\maketitle

% \begin{center}
% In collaboration with:\\
%   {\TheCollaborators}
% \end{center}
% \vspace{0.5cm}

% ---------------------------------------------
% ---------------------------------------------

\begin{abstract}
We propose an approximate interpolation construction and a practical coarsening algorithm for  algebraic multigrid (AMG) method when applied to $\Hcurl$ problems. Building on the generalized AMG framework, 
we introduce an interior/exterior splitting that yields both a refinement-based and a fully algebraic construction of the interpolation. 
The refinement-based approach follows a geometric hierarchy, while the purely algebraic interpolation is constructed through a coarsening process 
that first coarsens a nodal dual problem and then builds coarse and fine variables using a matching algorithm. 
Under a uniform fine-space smoothing assumption, we prove a jump-independent two-grid convergence bound for canonical finite-element spaces.
We also establish an exact commuting relation for the unshifted construction. 
Two-dimensional numerical experiments demonstrate sparse multilevel hierarchies and robust convergence under strong coefficient jumps in the curl--curl term. Three-dimensional experiments provide a two-level proof of concept, although the interpolation operators remain dense. Developing a practical sparse three-dimensional multilevel method is left to future work.

  % A limit of this interpolation, however, is that the quality of thcoarse grid depends on the result of the nodal coarsening.
  
\end{abstract}

\begin{keywords}
  algebraic multigrid, ideal interpolation, \Hcurl, de Rham complex
\end{keywords}

\section{Introduction}\label{sec:intro}
Algebraic multigrid (AMG) is an algorithm for solving linear systems of equations typically derived from the discretization of a partial differential equation,
\begin{equation}
\label{eqn:axb}
A \Xu = \Xf,
\end{equation}
where $A \in \mathbb{R}^{n\times n}$ is a sparse matrix. AMG constructs a fast solver through a complementary process of relaxation and coarse-grid correction. The basic approach in AMG is to identify the smooth error (i.e., error not effectively reduced by relaxation) and then coarsen and build interpolation that directly approximates and treats such error.  For scalar systems, one typically assumes the constant represents smooth errors locally, and this assumption is the key to the design of classical AMG~\cite{RuStu1987,brandt1983automatic,brandt1985sparse}. 

However, classical injection-based AMG is inadequate for \Hcurl problems
with a large near-kernel. The curl--curl operator has a kernel consisting of discrete
gradients.  Consequently, pointwise relaxation does not effectively reduce all
locally oscillatory curl-free errors, while specially designed interpolation
is needed to transfer the gradient error to the coarse grids.  Early geometric
multigrid methods addressed this difficulty through discrete Helmholtz
decompositions and specialized Schwarz or hybrid smoothers, treating gradient
fields and their complement separately
\cite{hiptmair1998multigrid,arnold2000multigrid}.

Auxiliary-space methods approximate an \Hcurl field through the discrete
gradient of a scalar nodal space and an interpolation from a nodal vector
$H^1$ space.  The Hiptmair--Xu preconditioner uses these transfer operators to
reduce the edge problem to Poisson-type nodal auxiliary problems~\cite{HX07}.
Its parallel extension, AMS in hypre, uses standard AMG for the auxiliary
solves~\cite{kolev2009ams,falgout2006hypre}.  For a fixed, simple two-region
partition, published AMS experiments show only mild iteration growth even as
the coefficient contrast reaches $10^8$.  These results, while impressive, do not
establish independence from the geometry or scale of general jump
distributions.

Structure-preserving AMG methods use the commuting structure of
the de Rham complex to algebraically construct multigrid hierarchies.  Reitzinger and
Sch\"oberl~\cite{reitzinger2002algebraic} first constructed coarse edge spaces
from nodal aggregates while preserving commuting relations.  Bochev et
al.~\cite{bochev2003improved} introduced smoothed edge interpolation, and Hu et
al.~\cite{hu2006hindependent} generalized it to a two-parameter transfer that
substantially improved mesh independence.  Tuminaro and
Glusa~\cite{tuminaro2025structure} more recently replaced the previous
smoothing formula by constrained energy minimization, allowing general nodal
AMG interpolation and strength-based sparsity patterns while retaining the commuting
relation.  Element-based AMGe in ParELAG~\cite{kalchev2023parelag}
similarly constructs compatible nested spaces forming an exact de Rham
sequence on every level.  These aggregation-based methods exhibit
mesh-independent or nearly mesh-independent performance, including for large
conductivity variations.  Their model jump tests, however, place the
discontinuity in the conductivity, or mass, term and do not establish
independence from jumps in the coefficient multiplying the curl--curl term.

On the coefficient-jump side, Jones and Lee~\cite{jones2006variable} developed
operator-dependent geometric multigrid with local harmonic extension tailored
to discontinuities in both the curl--curl and mass coefficients.  Their method
showed robust performance for such jumps, but relied on a structured geometric
hierarchy.

In this work, we develop an approximation of ideal interpolation
tailored to \Hcurl problems. The novelty of this work stems from its 
robustness with respect to jumps in the coefficient multiplying the
curl--curl term.  Derived from the generalized AMG (GAMG)
framework~\cite{FalgoutGAMG}, the interpolation approximates ideal
interpolation by coefficient-dependent harmonic extensions from two different  interior/exterior splitting schemes. The first is refinement-based, while the second uses fully algebraic coarsening through a standard unweighted nodal splitting, without coefficient-dependent strength filtering or alignment with jump interfaces. For both constructions with
the stated stability and smoothing assumptions, we prove that the
associated two-grid convergence bound, for canonical finite-element
fine spaces and exact coarse solves, is independent of the magnitude and distribution of jumps in the curl--curl coefficient. These results apply in both two and three dimensions. 
We perform numerical experiments with strong, spatially varying
curl--curl coefficient jumps.  In two dimensions,
both schemes produce sparse interpolation operators and demonstrate robust multilevel convergence. Three-dimensional two-level experiments show favorable convergence, but dense harmonic extensions currently limit practical multilevel use. Therefore, three-dimensional constructions should be viewed as a proof-of-concept framework. Developing practical sparse three-dimensional interpolations is outside the scope of this paper.
The exact commuting relation holds for the unshifted construction.  For
interpolation constructed from a shifted operator, we
measure the commuting defect in Section~\ref{sec:num}.

The remainder of the paper is organized as follows.  Section~\ref{sec:background}
reviews the curl--curl equation and the GAMG theory.
Section~\ref{sec:refinement} develops the refinement-based coarse variables and
establishes the jump-robust convergence and commuting structure results.
Section~\ref{sec:algebraic} presents the fully algebraic coarsening algorithm,
and Section~\ref{sec:num} reports the numerical results.  We conclude in
Section~\ref{sec:conc}.

\section{Background}\label{sec:background}
\subsection{The Curl--Curl Equation}\label{sec:curlcurl_equation}
Consider the shifted definite curl-curl equation
\begin{equation}\label{eqn:curl}
\begin{aligned}
\nabla \times\!\big(\mu(\Xx)\,\nabla \times \Xu\big) + \beta_0 \Xu &= \Xf \quad \text{in } \Omega,\\
\Xu \times \Xn &= 0 \quad \text{on } \partial\Omega,
\end{aligned}
\end{equation}
where $\Omega \subset \mathbb{R}^d$, $\mu(\Xx)$ is positive piecewise
constant and may exhibit arbitrarily large jumps across interfaces, and
$\beta_0>0$ is a spatially constant shift. We are primarily interested in
the regime $\beta_0\ll1$, where the gradient space is a large near-kernel. We assume that $\Omega$ simply connected with connected boundary.
Discretizing~\eqref{eqn:curl} with lowest-order N\'{e}d\'{e}lec edge elements~\cite{nedelec1980mixed,hiptmair2002finite} yields the SPD system
\begin{equation*}
    A \Xu = \Xf,\qquad A = A^{(s)}(\mu) + \beta_0 A^{(m)},
\end{equation*}
where $A^{(s)}(\mu)$ is the curl--curl stiffness matrix and $A^{(m)}$ is the unweighted mass matrix.
Essential boundary {\em degrees of freedom} (DoFs) are eliminated. The discrete gradient $G$ has signed incidence entries $0,\pm1$ mapping from the nodal space to the edge space.
For any symmetric positive-semidefinite matrix $B$, we use the energy seminorm
\[
    |\Xv|_B:=\bigl(\Xv^T B\Xv\bigr)^{1/2}.
\]
In particular, $|\cdot|_{A^{(s)}}$ is a seminorm because $A^{(s)}$ has the
nontrivial kernel described below. We retain $\|\cdot\|_B$ when $B$ is
positive definite.
By the de~Rham complex, the stiffness matrix satisfies
\begin{equation}\label{eq:deRham}
    \ker\!\big(A^{(s)}(\mu)\big) = \mathrm{range}(G),
\end{equation}
When $\beta_0\to0$, the $O(n)$-dimensional gradient kernel of $A^{(s)}$
forms a near-kernel of $A$. In this case, $A$ becomes increasingly ill-conditioned, and non-pointwise
smoothers~\cite{hiptmair1998multigrid} are needed for an optimal multigrid
method.

%% ==================================================================================================
\subsection{The Generalized AMG Theory}\label{sec:gamg_theory}
The two-grid AMG scheme is summarized by the error propagator
\begin{equation}\label{eq:tl}
E_{TG} = (I-M^{-1}A)(I-\Pi_A(P)),
\end{equation}
where $M$ defines the smoother, $P\in\mathbb{R}^{n\times n_c}$ is the
interpolation matrix, $A_c:=P^TAP$ is the Galerkin coarse-grid matrix, and
\[
    \Pi_A(P)=P A_c^{-1}P^T A
\]
is the $A$-orthogonal projection onto $\operatorname{range}(P)$.  For
symmetric $A$, the generalized AMG two-grid theory gives
\cite{FalgoutGAMG,sharp_theory_2005,G-BAMG-2018}
\begin{equation}\label{eqn-tg-measure}
\begin{split}
  \|E_{TG}\|_A^2 &\leq 1-\frac{1}{K},\\
  K:=K(PR)
  &=\sup_{\Xe}\frac{\|(I-PR)\Xe\|_{\XMt}^2}{\|\Xe\|_A^2}\geq 1,
\end{split}
\end{equation}
where $\XMt=M(M+M^T-A)^{-1}M^T$ is the symmetrized smoother and
$R:\mathbb{R}^n\to\mathbb{R}^{n_c}$ satisfies $RP=I_c$, so that $PR$ is a
projection onto $\operatorname{range}(P)$.

The GAMG result~\eqref{eqn-tg-measure} begins with the orthogonal decomposition
\[
    \mathbb{R}^n=\operatorname{span}(S)\oplus
    \operatorname{span}(R^T),
\]
where the columns of $S\in\mathbb{R}^{n\times n_s}$ form a basis for the
fine variables and the columns of $R^T\in\mathbb{R}^{n\times n_c}$ form a
basis for the coarse variables.  Under the GAMG normalization
$RS=0$, $S^TS=I_s$, and $RR^T=I_c$,
\begin{equation}\label{eqn:split}
    \Xu=\ZQS\Xu+\ZQR\Xu
       =S\Xu_s+R^T\Xu_c,\qquad
    \ZQS=\ZS\ZS^T,\qquad \ZQR=\ZR^T\ZR.
\end{equation}
For fixed $R$ (equivalently, fixed $S$), the corresponding ideal
interpolation that minimizes~\eqref{eqn-tg-measure} is
\begin{equation}\label{eqn-ideal}
\begin{split}
  P_{\star}:=\arg\min_{P:\,RP=I_c}K(PR)
  &=
  \left[\,S\ \ R^T\,\right]
  \left[\begin{array}{c}
    -(S^TAS)^{-1}(S^TAR^T)\\ I
  \end{array}\right]\\[1ex]
  &=(I-S(S^TAS)^{-1}S^TA)R^T.
\end{split}
\end{equation}
Thus, ideal interpolation is a global harmonic extension in the new bases
that minimizes the GAMG two-grid convergence bound for the fixed smoother.

\subsection{AMG Interpolation for Curl--Curl}\label{sec:amg_curlcurl}
Our construction is motivated by the ideal interpolation
operator~\eqref{eqn-ideal}.  In classical AMG, after ordering the unknowns
into fine and coarse DoFs the matrices defining the fine and
coarse variables take the coordinate form
\begin{equation}\label{eq:classicalRS}
    R=\begin{bmatrix}0&I\end{bmatrix},
    \qquad
    S=\begin{bmatrix}I\\0\end{bmatrix}.
\end{equation}
Thus, the coarse variables are values at a selected subset of the original
DoFs.  For curl--curl problems, such a coordinate selection
does not in general preserve the signed discrete gradients.  Consequently, even with an
effective block smoother, interpolation based on these coarse variables may
fail to transfer the large gradient near-kernel and can lead to convergence
stagnation.

Structure-preserving AMG methods address this issue by constructing edge and
nodal transfers that satisfy the commuting relation~\cite{reitzinger2002algebraic,bochev2003improved, hu2006hindependent,tuminaro2025structure}.
Namely,
\[
\begin{tikzcd}[column sep=large, row sep=large]
\mathbb{R}^{m_c}
  \arrow[r, "G_c"]
  \arrow[d, "P_n"']
&
\mathbb{R}^{n_c}
  \arrow[d, "P"]
\\
\mathbb{R}^{m}
  \arrow[r, "G"']
&
\mathbb{R}^{n}
\end{tikzcd},
\]
where $m$, $m_c$ count fine and coarse nodal DoFs, $n$, $n_c$ count fine and coarse edge DoFs, $G_c$ is the coarse discrete gradient, and $P_n$ is a nodal interpolation. Equivalently, 
\begin{equation}\label{eq:commuting}
    PG_c=GP_n.
\end{equation}
This relation ensures that coarse nodal gradients are represented as fine-level
gradients.  It is an important structural requirement for multilevel recursion~\cite{reitzinger2002algebraic}, 
but it does not by itself guarantee robustness in the case of large varying coefficients.

Our approach, on the other hand, first chooses coarse variables that preserve appropriate
signed combinations of local gradient components. Second, we extend these coarse variables with
coefficient-dependent harmonic extension.  This construction approximates the global ideal interpolation
in~\eqref{eqn-ideal} and is able to establish stability with respect
to curl--curl coefficient jumps. 

\section{Refinement-based Coarse Variables}\label{sec:refinement}
In this section, we construct the coarse variables and the resulting interpolation operator by following the geometric refinement of the mesh. To motivate the construction of $R$ and $S$, we begin with the curl-curl problem on a uniform quadrilateral mesh shown in Figure~\ref{fig:R_orientation}. The arrows indicate the orientation of the fine edges. Each edge degree of freedom is the oriented edge integral of the tangential component. Following geometric refinement, we can partition these 12 DoFs into the 8 external DoFs highlighted in blue and the remaining 4 interior DoFs, as shown in Figure~\ref{fig:R_quad}. Specifically, each coarse edge is interpolated from two fine edges in the same dotted blue oval and they represent the support of the two nonzeros in a row of $R$. 

\begin{figure}[!htbp]
    \centering
    \begin{subfigure}[t]{0.21\textwidth}
        \centering
        \includegraphics[width=\linewidth]{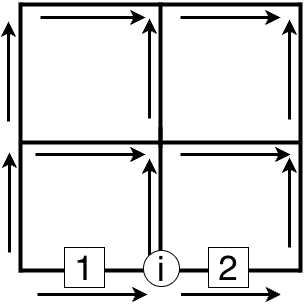}
        \caption{Nodal coarsening}
        \label{fig:R_orientation}
    \end{subfigure}
    \hfill
    \begin{subfigure}[t]{0.24\textwidth}
        \centering
        \includegraphics[width=\linewidth]{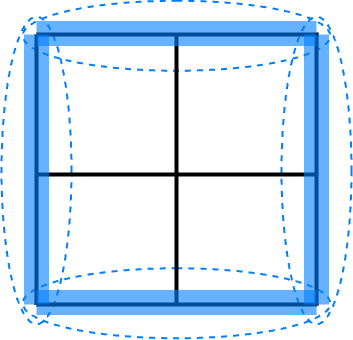}
        \caption{Pattern in $R$}
        \label{fig:R_quad}
    \end{subfigure}
    \hfill
    \begin{subfigure}[t]{0.26\textwidth} 
        \centering
        \includegraphics[width=\linewidth]{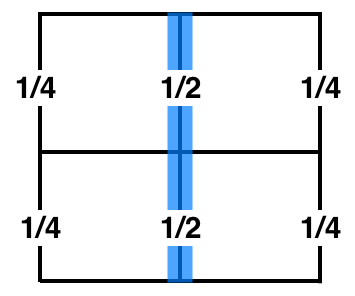}
        \caption{Vertical stencil}
        \label{fig:vertical_stencil}
    \end{subfigure}
    \hfill
    \begin{subfigure}[t]{0.22\textwidth} 
        \centering
        \includegraphics[width=\linewidth]{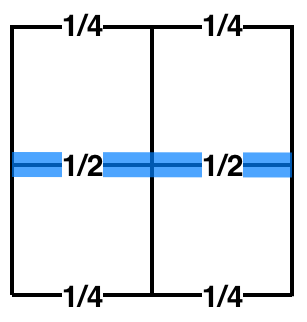}
        \caption{Horizontal stencil}
        \label{fig:horizontal_stencil}
    \end{subfigure}
    \caption{Curl-curl on uniform quadrilateral mesh.}
    \label{fig:curl-curl}
\end{figure}

We now focus on the bottom coarse edge in Figure~\ref{fig:R_orientation}, formed by the two fine edges DoFs labeled 1 and 2 meeting at the node $i$. The $i$-th column of the discrete gradient $G$, restricted to these two edge DoFs (after normalization), is 
\begin{equation}\label{eq:S_col}
    \frac{1}{\sqrt{2}} \begin{bmatrix}
        1 &-1
    \end{bmatrix}^T,
\end{equation}
with signs determined by the edge orientations in the figure. We assign this vector as a column of $S$, placing these entries on DoFs 1 and 2 and zeros elsewhere. Since $R$ and $S$ must be orthogonal, the corresponding row of $R$ must be the orthogonal complement of 
$\frac{1}{\sqrt{2}}\begin{bmatrix}
    1,-1
\end{bmatrix}^T$. We use the simplest choice 
\begin{equation}\label{eq:R_row}
     \frac{1}{\sqrt{2}}\begin{bmatrix}
        1& 1
    \end{bmatrix}.
\end{equation}
If one of the fine edges has an opposite orientation, the signs in both $R$ and $S$ must be flipped consistently according to $G$. To span the remaining interior space, we assign corresponding Euclidean basis vectors as columns in $S$. $S$ and $R$ admit the following block representations
\begin{equation}\label{eq:RS_splitting}
    S = 
    \begin{bmatrix}
        S_I & 
        S_E
    \end{bmatrix}, \qquad
    R = 
    \begin{bmatrix}
        0 & R_E
    \end{bmatrix},
\end{equation}
where the subscripts $I$ and $E$ denote the interior Euclidean bases and the exterior pairs from equations~\eqref{eq:S_col} and~\eqref{eq:R_row}, respectively. 
For example, in Figure~\ref{fig:R_quad}, order the four interior DoFs first, followed by the exterior pairs on the bottom, right, top, and left edges, each ordered along its orientation. Then
\begin{equation*}
    S=\begin{bmatrix}
        S_I &
        S_E
    \end{bmatrix}
    =\begin{bmatrix}I_4&0\\0&\widehat S_E\end{bmatrix},
    \qquad 
    R=\begin{bmatrix}
        0 &
        R_E
    \end{bmatrix},
\end{equation*}
where
\begin{equation*}
\widehat S_E=\frac{1}{\sqrt{2}}
\begin{bmatrix}
1&0&0&0\\-1&0&0&0\\
0&1&0&0\\0&-1&0&0\\
0&0&1&0\\0&0&-1&0\\
0&0&0&1\\0&0&0&-1
\end{bmatrix},
\qquad
R_E=\frac{1}{\sqrt{2}}
\begin{bmatrix}
1&1&0&0&0&0&0&0\\
0&0&1&1&0&0&0&0\\
0&0&0&0&1&1&0&0\\
0&0&0&0&0&0&1&1
\end{bmatrix}.
\end{equation*}
Thus $RS=0$, $S^TS=I_8$, and $RR^T=I_4$.

In this special case of Figure~\ref{fig:R_orientation}, where the two fine edges have consistent orientation, our choice of $R_E$ and $S_E$ coincides with the choice in Theorem 6.2 of~\cite{FalgoutGAMG}. In that setting, the coarse variables are defined purely to ensure orthogonality between $R$ and $S$, namely,
\begin{equation*}
     R_{E} = \begin{array}{c}
  \frac{1}{\sqrt{2}}\left[\begin{array}{ccc}   I &  I \end{array}\right]
  \end{array}, \qquad
  S_{E} = \begin{array}{c}
  \frac{1}{\sqrt{2}}\left[\begin{array}{ccc}  I &  -I \end{array}\right]
  \end{array}
\end{equation*}
However, this choice does not explicitly incorporate the orientation induced by the $G$. Consequently, the resulting interpolation doesn't preserve the commuting in Equation~\ref{eq:commuting} in general. In contrast, our construction preserves near-kernels, as we show in the following section.

\subsection{Sparse approximation of the ideal interpolation}
% Recall that the ideal interpolation operator associated with the splitting 
% $S = [\, S_I \; S_E \,]$ and $R^T$ can be written in the form
% \begin{equation}\label{eq:idealP_SR}
% P_\star 
% = 
% \left(I - S (S^T A S)^{-1} S^T A \right) R^T.
% \end{equation}
% In this representation, the columns of $S_I$ span the interior modes, 
% while the columns of $S_E$ represent oriented exterior difference modes arising from pairs of fine edges. Due to the orientation-aware construction described above, 
% the decomposition
% \[
% \mathrm{span}\{R_{E_G}\} \oplus \mathrm{span}\{S_E\}
% \]
% provides a local splitting of the exterior edge space into components that 
% preserve discrete gradients and components that are locally orthogonal to 
% the gradient space $\XRange{G}$. Since $\ker(K) = \XRange{G}$ for the curl--curl operator, preserving the near-kernel requires that the coarse space reproduce the gradient modes exactly. The operator $P_\star$ accomplishes this by eliminating both the interior modes in $\XRange{S_I}$ and the exterior difference modes in 
% $\XRange{S_E}$ through an exact $A$-harmonic extension.

% However, eliminating the full space $\XRange{S}$ generally produces a dense 
% operator, since $(S^T A S)^{-1}$ couples interior and exterior components. 
To obtain a sparse approximation of the ideal interpolation defined in~\eqref{eqn-ideal}, we retain $S_I$ and drop $S_E$ in $S$. Intuitively, we drop $S_E$ because the exterior DoFs are already represented by the coarse variables in $R^T$. Moreover, the truncation makes the resulting harmonic extension local while, as shown below, retain important stability properties. This leads to the sparse approximation
\begin{equation}\label{eq:sparseP}
P_{\mathrm{app}} 
:= 
\left(I - S_I (S_I^T A S_I)^{-1} S_I^T A \right) R^T = (I - \Pi_A(S_I))R^T,
\end{equation}
an $A$-harmonic extension into the interior of the gradient-preserving coarse variables.
For the two-dimensional refinement-based splitting, the exterior edges
separate the interiors of neighboring coarse elements, so
$S_I^TAS_I=A_{II}$ consists of independent local blocks.  The harmonic
extension can therefore be computed by local solves and yields sparse
interpolation.  The dependence of this locality on the splitting and the
spatial dimension is discussed in Section~\ref{sec:practicality}. 

Notation-wise, we write the interior/exterior splitting of $A$ and $G$ as 
\begin{equation*}
    A =
    \begin{bmatrix}
    A_{II} & A_{IE} \\
    A_{EI} & A_{EE}
    \end{bmatrix},
    \qquad G = 
    \begin{bmatrix}
    G_{I}  \\
    G_{E}
    \end{bmatrix}.
\end{equation*}
Let $G_S$ denote the discrete gradients whose exterior restrictions form the columns of $S_E$.  
Namely, 
\begin{equation}\label{eq:GS_lift}
    G_S=
    \begin{bmatrix}
        G_{I,S}\\
        S_E
    \end{bmatrix}
    =GJD,
    \qquad
    S_E=G_EJD
\end{equation}
where $J$ selects the corresponding nodal variables and $D$ is diagonal and contains the normalization factors. 
For the normalized length-two paths used in our construction, $D=I/\sqrt{2}$. 
In particular, $A^{(s)}G_S=0$ by~\eqref{eq:deRham}.

\subsection{Jump-independent two-grid convergence}
\label{sec:jump_robust_Papp}
We now prove a conditional two-grid convergence bound that is independent of
the coefficient variation $\mu (\Xx)$ in the curl--curl term.
Throughout this subsection, let $\{\mathcal T_\ell\}$ be a conforming, shape-regular, quasi-uniform family of affine simplicial or hexahedral meshes,
and assume the Galerkin coarse problem is solved exactly. See~\cite[Section~3.1]{ciarlet2002finite} for detailed definitions. 
Let $h_\ell=\max_{T\in\mathcal T_\ell}\operatorname{diam}(T)$ denote the maximum element diameter on level $\ell$.
For any $\Xv\in\mathbb R^n$, we use the block notation
$\Xv=[\Xv_I^T\;\Xv_E^T]^T$, where the subscripts $I$ and $E$ denote
its interior and exterior components. 

Consider the symmetrized smoother
\[
    \widetilde M=M(M+M^T-A)^{-1}M^T,
\]
where $M$ is nonsingular and $M+M^T-A$ is positive definite.
For a fixed $\beta_0>0$, we assume that there is a constant
$C_{\mathrm{sm}}>0$, independent of $\ell$ and $\mu$, such that
\begin{equation}
\label{eq:uniform_fine_smoothing}
    \|\Xz\|_{\widetilde M}^2
    \le C_{\mathrm{sm}}\|\Xz\|_A^2,
    \qquad \Xz\in\ker R.
\end{equation}
We adopt \eqref{eq:uniform_fine_smoothing} as an explicit assumption on the
fixed smoother to isolate the interpolation contribution.

\begin{lemma}
\label{lem:mass_stable_gradient_removal}
Let $G_S$ be defined in~\eqref{eq:GS_lift}. There is a constant \(C_T>0\), independent of the mesh size,
\(\beta_0\), and \(\mu\), such that
\begin{equation}
\label{eq:mass_stable_gradient_removal}
    \|\Xe-G_SS_E^T\Xe_E\|_{A^{(m)}}
    \le C_T\|\Xe\|_{A^{(m)}},
    \qquad \Xe\in\mathbb R^n.
\end{equation}
\end{lemma}

\begin{proof}
Canonical edge bases with shape regularity and quasi-uniformity give
\begin{equation}
\label{eq:mass_euclidean_equivalence}
    c_m h_\ell^{d-2}\|\Xv\|_2^2
    \le \|\Xv\|_{A^{(m)}}^2
    \le C_m h_\ell^{d-2}\|\Xv\|_2^2,
    \qquad
    \Xv\in\mathbb R^n,
\end{equation}
where $c_m,C_m>0$ are independent of $h_\ell$, $\mu$, and $\beta_0$. 
See~\cite[Lemma~(3.12) and equation~(3.37)]{hiptmair2002finite} for details.
Moreover, shape regularity and quasi-uniformity give a uniform bound $N_*$ on the number of
edges incident to any selected vertex.  Since \(J\) selects distinct nodal variables,
\[
    \|GJ\|_\infty\le 2,
    \qquad
    \|GJ\|_1\le N_*.
\]
Consequently,
\[
    \|GJ\|_2
    \le \sqrt{\|GJ\|_1\|GJ\|_\infty}
    \le \sqrt{2N_*}.
\]
By construction, $\|D\|_2=1/\sqrt{2}$ and $\|S_E^T\|_2 = 1$.
Combining these estimates yields
\begin{equation}
\label{eq:gradient_removal_euclidean_bound}
    \|G_SS_E^T\Xe_E\|_2
    \le \|GJ\|_2\|D\|_2\|S_E^T\|_2\|\Xe_E\|_2
    \le \sqrt{N_*}\|\Xe\|_2.
\end{equation}
Combining~\eqref{eq:mass_euclidean_equivalence} and
\eqref{eq:gradient_removal_euclidean_bound} gives
\begin{align*}
    \|G_SS_E^T\Xe_E\|_{A^{(m)}}^2
    &\le C_m h_\ell^{d-2}\|G_SS_E^T\Xe_E\|_2^2\\
    &\le C_m h_\ell^{d-2}N_*\|\Xe\|_2^2\\
    &\le \frac{C_m N_*}{c_m}\|\Xe\|_{A^{(m)}}^2.
\end{align*}
The triangle inequality therefore gives
\[
    \|\Xe-G_SS_E^T\Xe_E\|_{A^{(m)}}
    \le C_T \|\Xe\|_{A^{(m)}},
\]
where $C_T = (1+\sqrt{N_*C_m/c_m})$.
\end{proof}

\begin{lemma}
\label{lem:Papp_minimum_energy}
For every \(\Xx_c\in\mathbb R^{n_c}\),
\[
    P_{\mathrm{app}}\Xx_c
    =\underset{\substack{\Xv\in\mathbb R^n\\\Xv_E=R_E^T\Xx_c}}
      {\operatorname{argmin}}\;\|\Xv\|_A^2.
\]
Moreover, 
\begin{equation}\label{eq:RPapp_right_inverse}
    RP_{\rm app}=I_c.
\end{equation}
\end{lemma}

\begin{proof}
Direct computation from~\eqref{eq:sparseP} gives
\[
    P_{\mathrm{app}}
    =
    \begin{bmatrix}
        -A_{II}^{-1}A_{IE}R_E^T\\
        R_E^T
    \end{bmatrix}.
\]
For arbitrary $\Xv$, we have
\begin{align*}
    \|\Xv\|_A^2
    &=
    \left\|\Xv_I+A_{II}^{-1}A_{IE}\Xv_E\right\|_{A_{II}}^2
      +\|\Xv_E\|_{X_E(A)}^2,
\end{align*}
where
\[
    X_E(A)=A_{EE}-A_{EI}A_{II}^{-1}A_{IE}.
\]
The first term is minimized uniquely by
\(\Xv_I=-A_{II}^{-1}A_{IE}\Xv_E\).  Therefore,
\begin{align*}
\underset{\substack{\Xv\in\mathbb R^n\\\Xv_E=R_E^T\Xx_c}}
{\operatorname{argmin}}\;\|\Xv\|_A^2
&=
\begin{bmatrix}
-A_{II}^{-1}A_{IE}R_E^T\Xx_c\\
R_E^T\Xx_c
\end{bmatrix}
=P_{\mathrm{app}}\Xx_c.
\end{align*}
Finally, since $R=[\,0\;R_E\,]$,
\[
RP_{\rm app}=R_ER_E^T=I_c.
\]
\end{proof}

\begin{theorem}
\label{thm:Papp_jump_independent_energy}
Let
\[
    Q_{\mathrm{app}}=P_{\mathrm{app}}R.
\]
\begin{equation}
\label{eq:Papp_energy_stability_Q}
    \|Q_{\mathrm{app}}\Xe\|_A^2
    \le \eta_{\mathrm{app}}\|\Xe\|_A^2,
    \qquad
    \eta_{\mathrm{app}}=\max\{1,C_T^2\},
\end{equation}
where \(\eta_{\mathrm{app}}\) is $\mu$-independent.
Equivalently,
\begin{equation}
\label{eq:Papp_energy_stability_IQ}
    \|(I-Q_{\mathrm{app}})\Xe\|_A^2
    \le \eta_{\mathrm{app}}\|\Xe\|_A^2.
\end{equation}

\end{theorem}

\begin{proof}
Let
\[
    \widehat{\Xe} =\Xe-G_SS_E^T\Xe_E.
\]
Because the exterior block of \(G_S\) is \(S_E\),
\[
    \widehat{\Xe}_E
    =\Xe_E-S_ES_E^T\Xe_E
    =R_E^TR_E\Xe_E.
\]
On the other hand, applying Lemma~\ref{lem:Papp_minimum_energy} with
\(\Xx_c=R\Xe=R_E\Xe_E\) gives
\begin{equation}
\label{eq:Qapp_competitor}
    \|P_{\mathrm{app}}\Xx_c\|_A^2 
    = \|Q_{\mathrm{app}}\Xe\|_A^2
    \le \|\widehat{\Xe}\|_A^2.
\end{equation}
Since $A^{(s)}G_S = 0$,
\begin{equation}
\label{eq:curl_energy_unchanged}
    |\widehat{\Xe}|_{A^{(s)}(\mu)}^2
    =|\Xe-G_SS_E^T\Xe_E|_{A^{(s)}(\mu)}^2
    =|\Xe|_{A^{(s)}(\mu)}^2.
\end{equation}
Finally, combining \eqref{eq:Qapp_competitor}, \eqref{eq:curl_energy_unchanged} and Lemma~\ref{lem:mass_stable_gradient_removal}
gives the desired bound~\eqref{eq:Papp_energy_stability_Q}:
\begin{align*}
    \|Q_{\mathrm{app}}\Xe\|_A^2
    &\le \|\widehat{\Xe}\|_A^2\\
    &=\widehat{\Xe}^{T}\bigl(A^{(s)}(\mu)+\beta_0 A^{(m)}\bigr)\widehat{\Xe}\\
    &=|\widehat{\Xe}|_{A^{(s)}(\mu)}^2
      +\beta_0\|\widehat{\Xe}\|_{A^{(m)}}^2\\
    &\le |\Xe|_{A^{(s)}(\mu)}^2
      +\beta_0 C_T^2\|\Xe\|_{A^{(m)}}^2\\
    &\le \max\{1,C_T^2\}
      \left(|\Xe|_{A^{(s)}(\mu)}^2
      +\beta_0\|\Xe\|_{A^{(m)}}^2\right)\\
    &=\max\{1,C_T^2\}\|\Xe\|_A^2.
\end{align*}
Since $RP_{\rm app}=I_c$, $Q_{\rm app}$ is a projection.
Therefore, $\|I-Q_{\rm app}\|_A=\|Q_{\rm app}\|_A$ and~\eqref{eq:Papp_energy_stability_IQ} follows.
\end{proof}

\begin{corollary}
\label{thm:Papp_jump_independent_K}
Under the smoothing assumption~\eqref{eq:uniform_fine_smoothing},
\begin{equation}
\label{eq:jump_independent_two_grid_bound}
    \|E_{TG}\|_A^2
    \le
    1-\frac{1}{C_{\mathrm{sm}}\max\{1,C_T^2\}}.
\end{equation}
\end{corollary}

\begin{proof}
Let
\[
    \Xu=(I-Q_{\mathrm{app}})\Xe.
\]
By~\eqref{eq:RPapp_right_inverse},
\[
    R\Xu
    =R(I-P_{\mathrm{app}}R)\Xe
    =R\Xe-RP_{\mathrm{app}}R\Xe
    =0.
\]
Thus $\Xu\in\ker R$, and
\eqref{eq:uniform_fine_smoothing} applies.  Together with
Theorem~\ref{thm:Papp_jump_independent_energy}, this gives
\begin{align*}
    \|(I-Q_{\mathrm{app}})\Xe\|_{\widetilde M}^2
    &\le C_{\mathrm{sm}}
       \|(I-Q_{\mathrm{app}})\Xe\|_A^2\\
    &\le C_{\mathrm{sm}}\eta_{\mathrm{app}}\|\Xe\|_A^2.
\end{align*}
Taking the supremum over $\Xe\ne0$ and applying
\eqref{eqn-tg-measure} with $P=P_{\mathrm{app}}$ proves
the result.
\end{proof}

Corollary~\ref{thm:Papp_jump_independent_K} establishes a uniform and jump-independent two-grid bound for canonical finite-element spaces. 
For ideal interpolation, $I-P_\star R=\Pi_A(S)$ is the $A$-orthogonal projection onto $\ker R$.
Thus, the smoothing assumption~\eqref{eq:uniform_fine_smoothing} gives
\[
 \|(I-P_\star R)\Xe\|_{\widetilde M}^2
 \le C_{\mathrm{sm}}\|\Pi_A(S)\Xe\|_A^2
 \le C_{\mathrm{sm}}\|\Xe\|_A^2.
\]
Consequently,~\eqref{eqn-tg-measure} yields a sharper jump-independent bound for $P_\star$
\[
 K(P_\star R)\le C_{\mathrm{sm}},\qquad
 \|E_{TG,\star}\|_A^2\le 1-\frac{1}{C_{\mathrm{sm}}},
\]
where $E_{TG,\star}$ denotes the two-grid error propagator with $P_\star$.

\subsection{Commuting relation}

The jump-independent convergence result above applies directly to the
shifted SPD operator.  The commuting relation is a separate structural
property of the corresponding unshifted interpolation.  Define the coarse
gradient by
\begin{equation}\label{eq:Gc}
    G_c := R G I_\mathcal{C},
\end{equation}
where $I_\mathcal{C}$ injects the selected coarse nodal variables.  For the
unshifted curl--curl operator $A^{(s)}$, define
\[
P_{\mathrm{app}}^s
:=
\left(I-S_I(S_I^TA^{(s)}S_I)^\dagger S_I^TA^{(s)}\right)R^T.
\]

\begin{theorem}
\label{thm:commuting_Papp}
There exists a nodal interpolation operator $P_n$ such that the commuting relation holds
\begin{equation}\label{eq:commuting_Papp}
P_{\mathrm{app}}^sG_c = GP_n.
\end{equation}
\end{theorem}

\begin{proof}
By~\eqref{eq:deRham}, it suffices to show that the columns of
$P_{\mathrm{app}}^sG_c$ lie in $\ker(A^{(s)})=\operatorname{range}(G)$.
A direct block computation yields 
\[
A^{(s)}P_{\mathrm{app}}^sG_c
=
\begin{bmatrix}
0\\
X_E(A^{(s)})\,R_E^TR_E\,G_E I_\mathcal{C}
\end{bmatrix},
\]
where $X_E(A^{(s)})=A^{(s)}_{EE}-A^{(s)}_{EI}(A^{(s)}_{II})^\dagger A^{(s)}_{IE}$.
Since $ R_E^TR_E = I - S_E S_E^T$, we can rewrite
\begin{align*}
    X_E(A^{(s)}) \, R_E^TR_E G_E I_\mathcal{C} 
    &=
    X_E(A^{(s)}) \, (I-S_ES_E^T) G_E I_\mathcal{C} \\
    &=
    X_E(A^{(s)}) \, G_E I_\mathcal{C} - X_E(A^{(s)}) \, S_E(S_E^T G_E I_\mathcal{C}).
\end{align*}
Since $A^{(s)}$ is positive semidefinite,
$\ker(A^{(s)}_{II})\subseteq\ker(A^{(s)}_{EI})$. The block equations of $A^{(s)}G=0$
therefore give
\[
    X_E(A^{(s)})G_E
    =A^{(s)}_{EI}\bigl((A^{(s)}_{II})^\dagger A^{(s)}_{II}-I\bigr)G_I=0.
\]
Moreover, $S_E=G_EJD$ by~\eqref{eq:GS_lift}, so
$X_E(A^{(s)})S_E=X_E(A^{(s)})G_EJD=0$.
Therefore, both terms of the second equality above vanish. 
\end{proof}

Although Theorem~\ref{thm:commuting_Papp} addresses only the unshifted interpolation
$P_{\mathrm{app}}^s$, it shows how the chosen splitting and harmonic extension
preserve discrete gradients and provides structural insight into the commuting
behavior of the method. For the interpolation
$P_{\mathrm{app}}$ constructed from the full operator $A=A^{(s)}+\beta_0A^{(m)}$
and used by the solver, we measure the commuting defect in
Section~\ref{sec:num}.  

%% ==================================================================================================

\section{Fully Algebraic Coarsening}\label{sec:algebraic}
In section~\ref{sec:refinement}, we introduced a refinement-based AMG algorithm that takes advantage of the hierarchical geometric structure to build a sparse interpolation. 
The key ingredient was an interior/exterior splitting together with a local harmonic extension that preserves near-kernel components. 
In this section, we extend this strategy to a fully algebraic setting. 
The interpolation continues to take the form in Eq.~(\ref{eq:sparseP}). 
Thus, the main task is to construct the interior/exterior splitting algebraically in order to define the coarse variables $R^T$ and interior Euclidean basis $S_I$.

The following steps summarize the two-level setup:
\begin{enumerate}
 \item Augment the gradient matrix $\widetilde G = \mathrm{augmentG}(G)$ to recover boundary nodes.
 \item Form and coarsen $A_{\widetilde G} = {\widetilde G}^T{\widetilde G}$.
 \item Construct the coarse and fine variables and compute the interpolation operator $P_{\mathrm{app}}$.
 \item Define the coarse-level gradient $G_c = R G I_\mathcal{C}.$
\end{enumerate}
We describe each step in detail below.

\subsection{Nodal coarsening }\label{sec:nodal}
Although the curl-curl operator is discretized with edge DoFs, we begin the algebraic construction by coarsening a dual nodal space. 
The motivation is that nodes are naturally associated with gradients. 
By nodal coarsening,  we implicitly select which gradient components are represented on the coarse grid. 
Moreover, nodal coarsening serves as an intermediate step that determines the endpoints of the support of the coarse-level gradient. 
These endpoints, in turn, define exterior edge variables by connecting edge DoFs between them. 

To formalize this idea, we define the nodal graph operator
\begin{equation*}
    A_G = G^TG.
\end{equation*}
Its off-diagonal sparsity pattern records the nodal adjacency encoded by the discrete gradient $G$.  
We coarsen this matrix graph directly to maintain connected coarse graphs and control operator complexity.  
A natural alternative of nodal graph operator is $G^TAG$. However, its sparsity pattern inherits the complexity of the edge operator $A$. Consequently, the coupling between $A$ and $G$ may cause large local components or disconnected regions that do not faithfully reflect the underlying nodal topology. 

Near a Dirichlet boundary, the eliminated nodal variables must be accounted for when constructing this graph.  
Recall that each row of $G$ corresponds to an edge and each column to a nodal DoF.  
An edge joining two nodes has two nonzeros, whereas an edge with only one endpoint has a single nonzero.  
To restore the missing endpoint in the graph, we append an artificial boundary-node column for each such row, 
with an entry opposite in sign to its existing nonzero.  The augmented gradient $\widetilde G$ thus retains the signed-incidence structure, 
and the nodal graph operator
\begin{equation*}
    A_{\widetilde G} = {\widetilde G}^T{\widetilde G}
\end{equation*}
encodes adjacency among interior and restored boundary endpoints.
Algorithm~\ref{alg:augmentG} describes this augmentation.  
Here, the artificial boundary-node indices $\mathcal{B}$ serve only to restore boundary connectivity. Denote
$\Xe_i\in\mathbb{R}^{n}$ the $i$-th Euclidean basis vector.

\begin{algorithm}[H]
\caption{augment $G$ with boundary-node columns}\label{alg:augmentG}
\begin{algorithmic}[1]
\algnewcommand{\Initialize}[1]{%
\Statex \hspace*{\algorithmicindent}\parbox[t]{.8\linewidth}{\raggedright #1}
}
\Function{augment$G$}{$G$}
\Initialize{$\widetilde G = G$, \ $m=\text{number of columns of }G$, \ $k=0$}
\For{$i = 1,2,\dots,\text{number of rows of }G$}
    \If{$\mathrm{nnz}(G(i,:)) = 1$} \Comment{$i$ is a boundary edge-row}
        \State Let $\delta$ be the unique nonzero value in row $G(i,:)$
        \State $k \gets k + 1$
        \State{$\widetilde G = \begin{bmatrix}\widetilde G & -\delta\Xe_i\end{bmatrix}$}

    \EndIf
\EndFor
\State $\mathcal{B} := \{m+1,\,m+2,\,\dots,\,m+k\}$ 
\State \Return $\widetilde G,\ \mathcal{B}$
\EndFunction
\end{algorithmic}
\end{algorithm}

We now select coarse nodal unknowns by performing a classical unweighted Ruge--St\"uben (RS) C/F splitting~\cite{RuStu1987} on the graph of $A_{\widetilde G}$. All graph neighbor are treated as connected without coefficient-dependent strength filtering or jump-interface information. 
Hence, the nodal coarsening only depends on graph connectivity alone. 
In the scope of this paper, we mark all nearest neighbors of coarse nodes as fine nodes. 
The subsequent construction of the exterior/interior splitting is formulated consistently with this choice. 
More aggressive nodal coarsening would require a modified strategy for constructing exterior paths and is therefore left for future research. 
Algorithm~\ref{alg:CF_bdry_AG} describes this coarsening process. 
We denote the set of neighbors of a DoF $b$ in the graph of $A$ by $\mathcal{N}_{A}(b) := \{j \neq b: (A)_{bj} \neq 0\}$.

\begin{algorithm}[H]
\caption{C/F splitting for $A_{\widetilde G}$ with prescribed boundary nodes}\label{alg:CF_bdry_AG}
\begin{algorithmic}[1]
\algnewcommand{\Initialize}[1]{%
\Statex \hspace*{\algorithmicindent}\parbox[t]{.8\linewidth}{\raggedright #1}
}
\Function{nodalCF}{$\widetilde G$, $\mathcal{B}$}
\Initialize{$A_{\widetilde G}=\widetilde G^{T}\widetilde G$, \ $\Xc=\mathcal{B}$, \ $\Xf=\emptyset$}
\For{$b \in \mathcal{B}$}
    \State $\Xf \gets \Xf \cup \mathcal{N}_{A_{\widetilde G}}(b)$ 
\EndFor
\State $(\Xc,\Xf)\gets \text{unweighted RS C/F splitting}\bigl(A_{\widetilde G};\, C_{\mathrm{ini}}=\Xc,\ F_{\mathrm{ini}}=\Xf\bigr)$ \label{alg:CF_ini}
\State $\XcG \gets \Xc \setminus \mathcal{B}$ \Comment{coarse set excluding boundary nodes}
\State $\XfG \gets \Xf$
\State \Return $\XcG,\ \XfG$
\EndFunction
\end{algorithmic}
\end{algorithm}

In this algorithm, the boundary-node indices $\mathcal{B}$ are prescribed as the initial coarse nodes, 
and their neighbors are marked as initial fine nodes (line~\ref{alg:CF_ini}). 
We then apply RS C/F splitting to the remaining interior nodal DoFs. 
Setting $\mathcal{B}$ as coarse nodes ensures that algebraic exterior paths can extend naturally to the boundary whenever needed in the later edge construction. 
Moreover, this choice helps maintain a structured and geometrically consistent exterior/interior splitting.

% \textcolor{red}{Up until this point, we have made no assumptions on the origin of the problem, only that the near-kernel matrix $K$ has been computed. } 
 
 % Then, having computed an appropriate splitting of the fine and coarse variables, as determined by coarsening and fast convergence of compatible relaxation (CR)~\cite{ABrandt_2000}, we compute an associated local harmonic approximation to ideal interpolation.  

% For instance, in the case of the curl-cur problem, the convergence factor of the two-grid method approaches 1 asymptotically when using the classical ideal interpolation as shown in detail in the numerical experiments in Section~\ref{sec:num}.

\subsection{Algebraic construction of coarse and fine variables}\label{sec:algebraic_p}
Having obtained the coarse nodal set $\XcG$, we now construct the coarse and fine variables in the edge space, namely $R$ and $S$. 
The goal is to algebraically construct the interior/exterior splitting introduced in Section~\ref{sec:refinement}. 
We achieve this by using the nodal coarsening as an auxiliary tool for pairing edge DoFs. 

In particular, $\XcG$ determines the endpoints of algebraic length-two paths. 
For each pair of coarse nodes $i,j \in \XcG$ that are distance-two neighbors in the nodal graph induced by $A_{\widetilde G}$, 
we identify a mutual distance-one neighbor $k$. The edge path connecting $i\to k\to j$ select two fine edge DoFs, 
denoted $\{DoF_1,DoF_2\}$. We then form a coarse edge variable, or equivalently a row in $R$, 
by averaging these two DoFs with signs consistent with the orientation stored in $G$. 
All remaining edge DoFs are classified as the interior and form a Euclidean basis $S_I$. 
This completes the interior/exterior splitting of the edge space. Algorithm~\ref{alg:p} summarizes the full construction.
Here, $\mathrm{supp}(\Xr):=\{\ell\mid r_\ell\ne0\}$ denotes the set of indices of the nonzero entries of $\Xr$.
\begin{algorithm}[!htbp]
\caption{construct the interpolation $P$}\label{alg:p}
\begin{algorithmic}[1]
\algnewcommand{\Initialize}[1]{%
\Statex \hspace*{\algorithmicindent}\parbox[t]{.8\linewidth}{\raggedright #1}
}
\Function{form$P$}{$A$, $G$}
\Initialize{$R=\emptyset,\ \Xdof=\{1,\dots,n\},\ \mathcal{K}=\emptyset$}
\State{$\widetilde G,\ \mathcal{B} \gets \mathrm{augmentG}(G)$}
\State{$\XcG,\ \XfG \gets \mathrm{nodalCF}(\widetilde G,\mathcal{B})$}
\State{$\XcG \gets \XcG \cup \mathcal{B}$} \label{alg:boundary_path}\Comment{allow paths to boundary nodes}

\For{$(i,j)\in \XcG\times \XcG,\ i<j$}
  \If{$i,j$ are distance-two neighbors \textbf{and} $(i,j)\not\subset\mathcal{B}$}\label{alg:unique_path}
    \State $k\gets$ a mutual distance-one neighbor of $i,j$
    \If{$k\notin \mathcal{K}$} \label{alg:unique_gradient} \Comment{enforce unique distance-2 path}
        \State Select $\{DoF_1,DoF_2\}\subset\Xdof$ connecting $i\to k\to j$
        \State $\Xdof\gets \Xdof\setminus\{DoF_1,DoF_2\}$ \label{alg:distinct_pairing}
        \State $\mathcal{K}\gets \mathcal{K}\cup\{k\}$
        \State Construct $\Xr$ with $\mathrm{supp}(\Xr)=\{DoF_1,DoF_2\}$ and entries $\pm1/\sqrt{2}$
        \State $R^T \gets [\, R^T \ \ \Xr \,]$
    \EndIf
  \EndIf
\EndFor

\State $S_I \gets I(:,\Xdof)$ \Comment{$S_I^TAS_I=A_{II}$}
\State \Return {$P_{\mathrm{app}} = R^T - S_I A_{II}^{-1}(S_I^T A R^T)$}
\EndFunction
\end{algorithmic}
\end{algorithm}

We now make several remarks on the design of Algorithm~\ref{alg:p}. 
First, the existence of distance-two paths is a direct consequence of the nodal coarsening  in Algorithm~\ref{alg:CF_bdry_AG}, 
which marks all nearest neighbors of coarse nodes as fine.
As a result, coarse nodes are typically separated by exactly one intermediate fine node, 
which makes selecting length-two paths $i\to k\to j$ available. Cost-wise, for quasi-uniform meshes, the size of the each two-ring neighborhood is uniformly bounded, so the second coarse neighbor search and the subsequent common-neighbor lookup require $O(1)$ work per coarse node, yielding overall linear complexity.
Second, line~\ref{alg:boundary_path} ensures that interior coarse nodes can form distance-two paths that extend to the boundary. 
Third, the condition $(i,j)\not\subset\mathcal{B}$ in line~\ref{alg:unique_path} ensures a unique path between coarse nodes. 
This reduces grid complexity while preserving a structured coarse-level gradient. 

We emphasize that the algebraic construction is designed to satisfy the orthogonality of $R$ and $S$.
Line~\ref{alg:distinct_pairing} enforces distinct edge DoFs in different rows of $R$, which preserves the orthogonality between $R$ and $S$. 
Moreover, line~\ref{alg:unique_gradient} ensures that only one path passes through one intermediate gradient. 
The set $\mathcal{K}$ keeps a record of which intermediate gradients, or equivalently intermediate fine nodes, have been used. 
This prevents overlap among the selected distance-two paths so that $RS = 0$ is preserved. 

Finally, we note that this fully algebraic interpolation also satisfies the two-grid bound in
Corollary~\ref{thm:Papp_jump_independent_K}, independently of coefficient
jumps in the curl--curl term, under the smoothing
assumption~\eqref{eq:uniform_fine_smoothing}.

To illustrate the construction, we apply Algorithm~\ref{alg:p} to curl-curl on a sample uniform triangular mesh. 
The results are shown in Figure~\ref{fig:tri_mesh_algebraic}. The red circles mark the coarse nodes, and blue denotes the distance-two paths that define the exterior variables in $R$. 
Figure~\ref{fig:tri_mesh_algebraic} shows that,
in general, the algebraic construction does not recover geometric refinement. 
Indeed, Figure~\ref{fig:tri_mesh_algebraic} differs from the geometric refinement shown in Figure~\ref{fig:tri_mesh_refinement}. Moreover, the sizes of each individual blocks vary. For example, the center two triangular blocks each has size 3, or contains 3 interior DOFs, whereas the upper-right most block has size 10.
Nevertheless, the algebraic splitting still produces highly structured gradients in the coarse grid. 
Moreover, as demonstrated in Section~\ref{sec:num}, the resulting convergence rates are very close when comparing the performance of using~\ref{fig:tri_mesh_refinement} versus~\ref{fig:tri_mesh_algebraic}.

\begin{figure}[!htbp]
    \centering
    \begin{subfigure}[t]{0.265\textwidth}
        \centering
        \includegraphics[width=\linewidth]{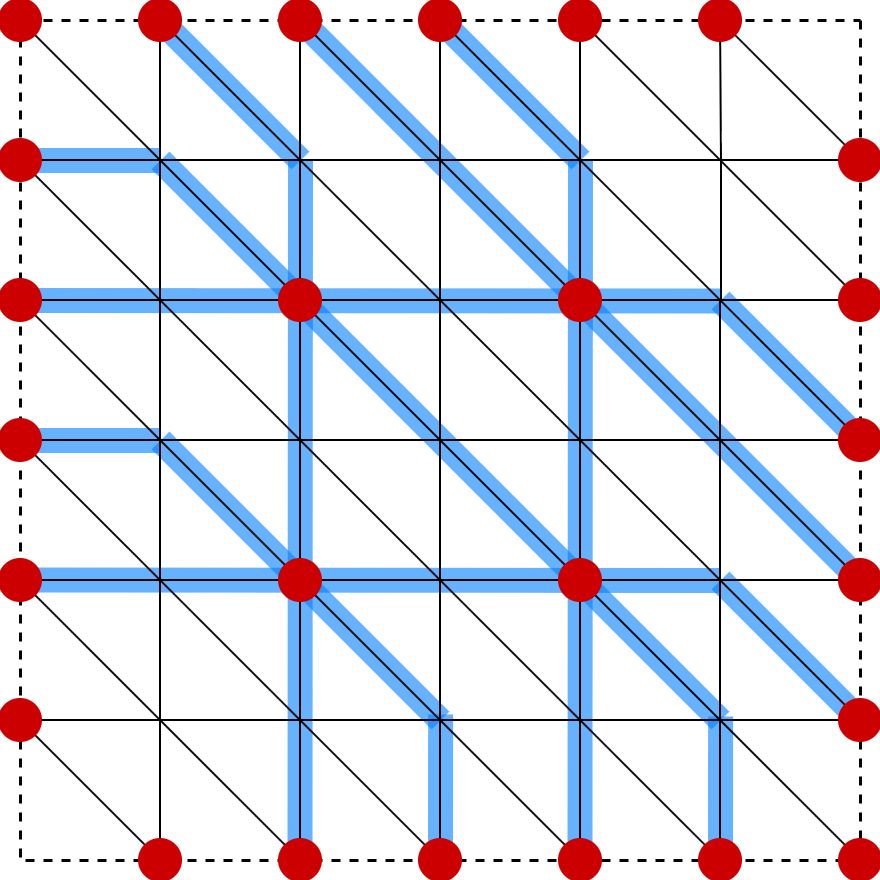}
        \caption{Algebraic coarsening}
        \label{fig:tri_mesh_algebraic}
    \end{subfigure}
    \hspace{0.15\textwidth}
    \begin{subfigure}[t]{0.265\textwidth} 
        \centering
        \includegraphics[width=\linewidth]{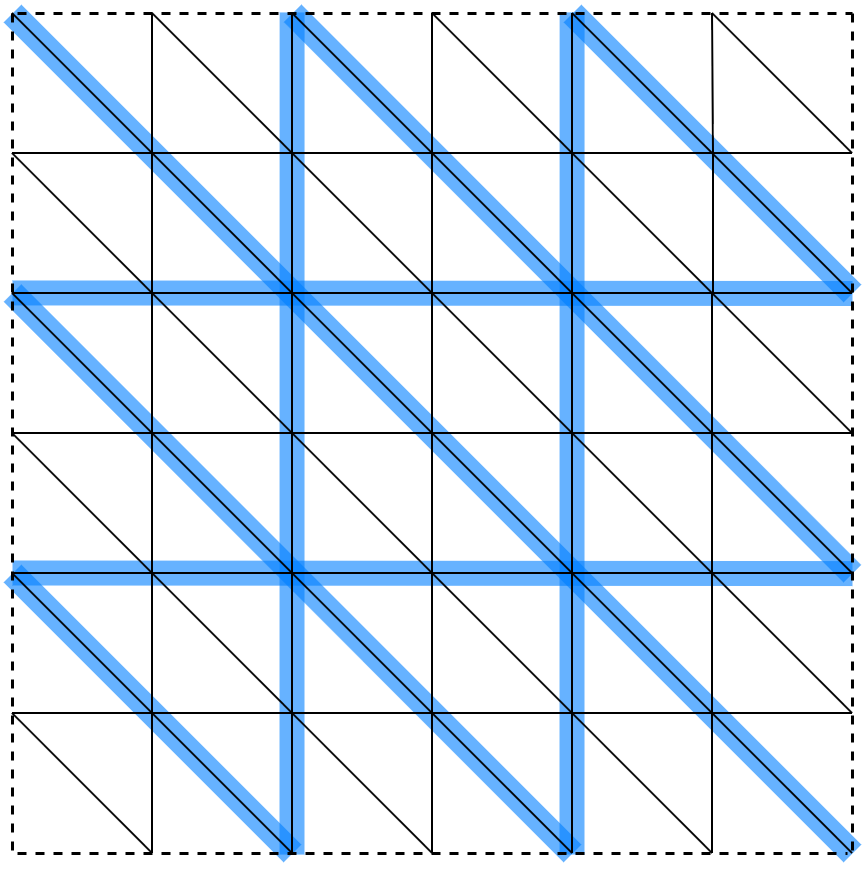}
        \caption{Refinement coarsening}
        \label{fig:tri_mesh_refinement}
    \end{subfigure}
    \caption{Algebraic coarsening of the curl--curl edge operator with Dirichlet boundary conditions on a uniform triangular mesh}
    \label{fig:curlcul_figure}
\end{figure}

Finally, we form the coarse-level gradient $G_c = RGI_\mathcal{C}$ in the multilevel hierarchy. 
For algebraic coarsening, $I_\mathcal{C} = I_\XcG$, the injection operator on the coarse nodes. Since the definition of $G_c$ is unchanged from that in Theorem~\ref{thm:commuting_Papp}, the theorem remains valid in the fully algebraic setting with the same proof.

\subsection{Sparsity and multilevel practicality}\label{sec:practicality}
The practical cost of both the refinement-based and fully algebraic
constructions is governed by the connectivity of the interior DOFs.
Writing~\eqref{eq:sparseP} as
\[
    P_{\mathrm{app}}
    =R^T-S_I A_{II}^{-1}S_I^TAR^T,
    \qquad A_{II}=S_I^TAS_I,
\]
shows that sparsity depends on the connected components of the interior graph. 
After reordering, $A_{II}$ is block diagonal with one block per
component, allowing independent solves on these blocks. 
Small interior components keep the extension local, whereas large components
can produce substantial fill-in in both $P_{\mathrm{app}}$ and the Galerkin
coarse operator $P_{\mathrm{app}}^TAP_{\mathrm{app}}$.

In two dimensions, the refinement-based exterior edges separate the interiors
of neighboring coarse elements, giving uniformly bounded interior blocks for
the fixed refinement patterns considered here. 
Figure~\ref{fig:triangle_coupling} illustrates one such block. The three red interior edge DOFs
are mutually coupled and confined to a single coarse triangle. 
Therefore, applying $A_{II}^{-1}$ is equivalent to independent solves with these small local blocks.
Algebraic path selection allows the same separation but does not guarantee uniformly bounded component sizes.
Both constructions yield practical multilevel hierarchies in the experiments of Section~\ref{sec:num}, 
although the growing operator complexity on the unstructured hierarchy in Table~\ref{tab:delaunay} illustrates this limitation of the algebraic construction.

In three dimensions, the selected exterior paths do not separate the interior graph into small local components. 
Fine-edge variables on coarse faces remain in the interior set and couple neighboring coarse elements, 
leading to dense harmonic extensions. 
Figure~\ref{fig:tetrahedral_coupling} illustrate this coupling. Each blue edge DOF couples to all three red interior edge DOFs on the red coarse face.
These common neighbors extend onto other blue coarse faces. As a result, the interior blocks are coupled through nonzero off-diagonal blocks in $A_{II}$, so we cannot apply $A_{II}^{-1}$ solve locally in 3D.
The storage and solve costs of the resulting dense Galerkin coarse operators limit multilevel use. 
The Ref and Alg experiments in Section~\ref{sec:num3d} therefore use two-level constructions as proofs of concept. 
A practical multilevel extension requires further localization or sparsification of the harmonic extension.

\begin{figure}[!htbp]
    \centering
    \begin{subfigure}[t]{0.25\textwidth}
        \centering
        \includegraphics[width=\linewidth]{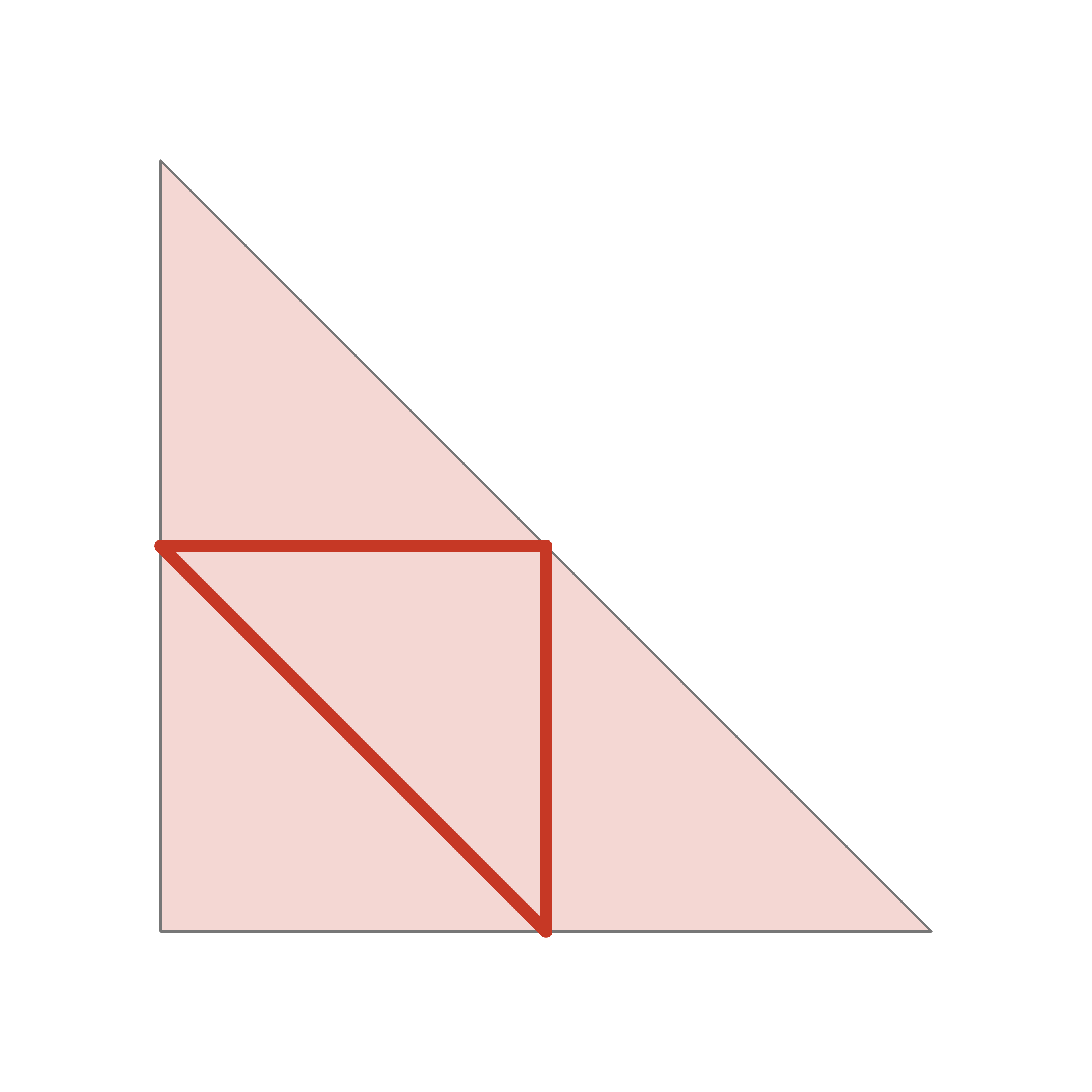}
        \caption{2D coupling}
        \label{fig:triangle_coupling}
    \end{subfigure}
    \hspace{0.10\textwidth}
    \begin{subfigure}[t]{0.25\textwidth} 
        \centering
        \includegraphics[width=\linewidth]{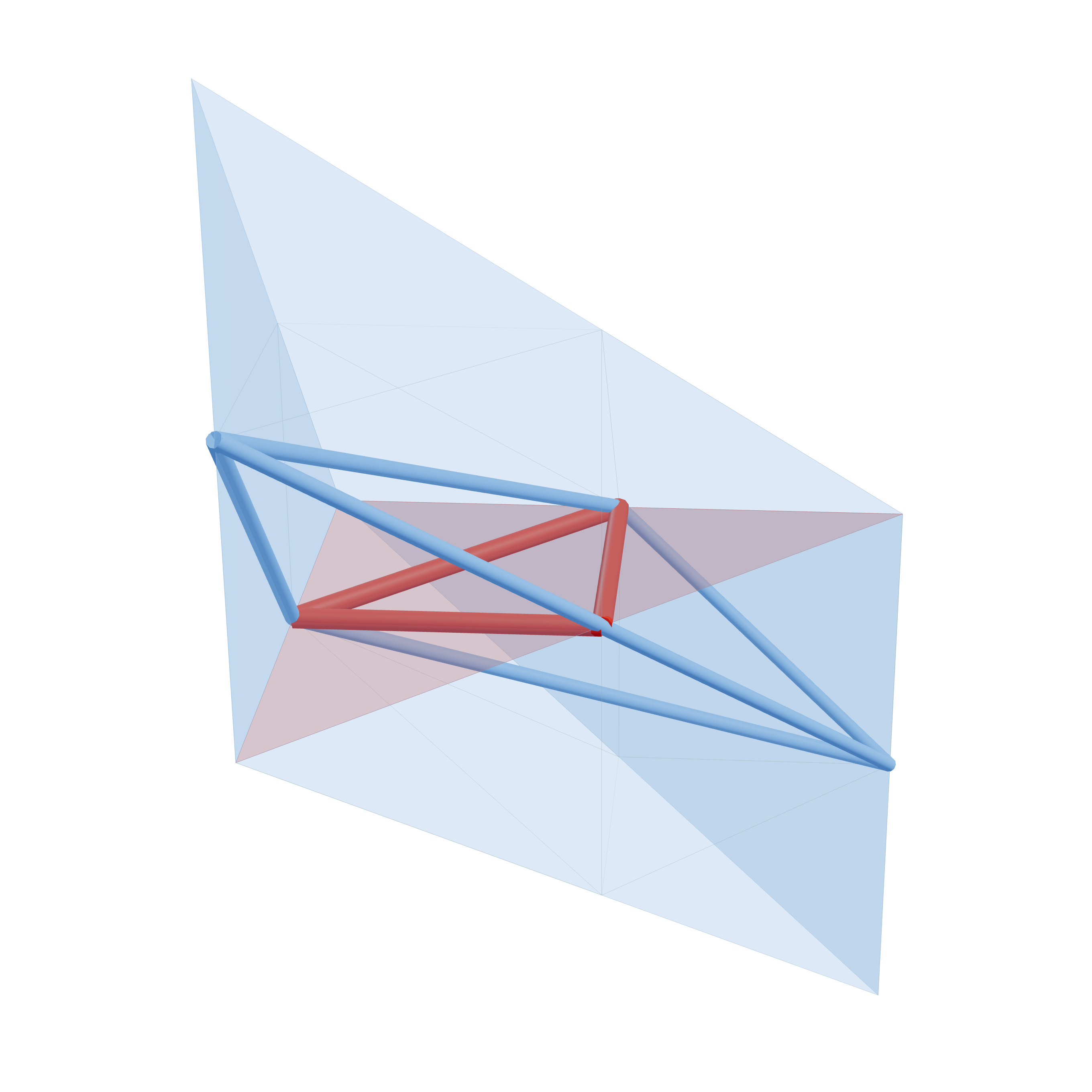}
        \caption{3D coupling}
        \label{fig:tetrahedral_coupling}
    \end{subfigure}
    \caption{The coupling of $A_{II}$ in 2D and 3D. Red edges denote the selected interior DOFs. Blue edges (if any) couple to the red edges. Shaded regions indicate the associated coarse triangles or faces.}
    \label{fig:curlcul_figure}
\end{figure}

\section{Numerical Results}\label{sec:num}
\subsection{Setup}
We solve the curl--curl equation~\eqref{eqn:curl} on $\Omega=(0,1)^2$ or $(0,1)^3$ with the
 boundary condition $\Xu\times\Xn=0$. We discretize the PDE with lowest-order
N\'ed\'elec edge elements. The mass shift $\beta_0=10^{-2}$ is fixed in every experiment.
The isotropic tests use $\mu(\Xx)\equiv 1$.  In every jump test, the curl--curl
coefficient has the form
\[
  \mu(\Xx)=
  \begin{cases}
    \mu_{\min}, & \Xx\in\Omega_{\mathrm{jump}},\\
    1, & \Xx\notin\Omega_{\mathrm{jump}},
  \end{cases}
  \qquad 0<\mu_{\min}\leq1,
\]
where the physical jump region $\Omega_{\mathrm{jump}}$ is specified for each
experiment. The coefficient contrast is the ratio of the largest to the smallest
coefficient value, $1/\mu_{\min}$. We evaluate the
proposed interpolation $P_{\mathrm{app}}$ with both the
geometric-refinement (denoted ``Ref'') and fully algebraic coarsening ("Alg"). We compare against
geometric N\'ed\'elec interpolation (``Ned'').  
These three methods use a symmetric V-cycle with one pre- and one post-smoothing step. 
For smoothers, we use forward multiplicative overlapping Schwarz relaxation followed by an $\ell_1$-Jacobi
iteration (OSS-$\ell_1$) for pre-smoothing, and backward OSS-$\ell_1$ for post-smoothing, respectively. Each OSS block consists of the edge DOFs in the support of a local discrete gradient. Namely, the support of the block associated with node $i$ is $\operatorname{supp}(G(:,i))$.
We allow at most 20 levels and 100 DOFs for the coarsest problem. We use a direct solve for the coarsest solve.
Denote $N_{\mathrm{level}}$ the total number of levels. All problems are generated in MFEM~\cite{anderson2021mfem}.

In addition, we benchmark AMS in hypre as a state-of-the-art for curl--curl systems, 
using its native auxiliary-space preconditioner and default algebraic hierarchies
\cite{kolev2009ams,falgout2006hypre}. In three dimensions, we also include
ParELAG through MFEM's \texttt{MultilevelHcurlSolver} miniapp, using its
native element-based AMGe hierarchy, Hiptmair pre- and post-smoothing,
$\ell_1$ Gauss--Seidel for the primary and auxiliary relaxations, and
AMS-preconditioned CG on the coarsest level
\cite{kalchev2023parelag,anderson2021mfem}. These solvers are compared as complete package-native configurations; they do not use the OSS-$\ell_1$ smoother. We estimate work per digit accuracy
(Section~\ref{sec:wpd}) to account for the different cycle costs. ParELAG is omitted from the two-dimensional tests because its available
driver implements only three-dimensional problems. 
% MueLu uses the MueLu
% Maxwell1 preconditioner with the Reitzinger constrained energy-minimization
% configuration \cite{tuminaro2025structure,heroux2005trilinos}.

We also tested MueLu's constrained energy-minimization solver
\cite{tuminaro2025structure}. Its default configuration generally performed worse than AMS on our
curl-coefficient jump problems. Its published
jump tests concern the mass coefficient, and adapting its parameters to jumps
in the curl--curl coefficient may require further tuning. We therefore omit
its results and retain AMS as the main solver comparison baseline.

The benchmark protocol reports PCG iterations.  For every method, the
initial guess is $\Xx_0=0$. 
The right-hand side $\Xb$ is generated randomly with a fixed seed and normalized.
Each method performs a PCG solve until the relative-residual
criterion
\[
  \frac{\|\Xb-A\Xx_k\|_2}{\|\Xb\|_2}\leq 10^{-6}
\]
is reached. For a
hierarchy with edge operators $A_0,\ldots,A_{L-1}$, we report the operator
complexity
\[
  \mathrm{OC}=\frac{\sum_{l=0}^{L-1}\operatorname{nnz}(A_l)}
                     {\operatorname{nnz}(A_0)}.
\]
We also measure the commuting defect without explicitly constructing a nodal
interpolation.  Recall that we define the coarse gradient algebraically by
\begin{equation*}
    G_c := R G I_\mathcal{C}.
\end{equation*}
At transfer level $l$, let $G_l$ and $G_{l+1}$ be the fine and coarse discrete gradients.  
We report the maximum relative Frobenius commuting defect over all transfer levels,
\[
  \delta_{\max}
  =\max_l
    \frac{\|(I-\Pi{(G_l)})P_lG_{l+1}\|_F}
         {\|P_lG_{l+1}\|_F}.
\]
Thus, $\delta_{\max}=0$ if~\eqref{eq:commuting} holds exactly on every level.

\subsubsection{Estimated work per digit accuracy}\label{sec:wpd}
We report the estimated solve work per digit accuracy (WPD),
following the approach of Ali et al.~\cite{ali2024constrained}.
One work unit is the cost of a matrix--vector product with the finest-grid operator $A_0$. 
An application of $B$ or $B^T$ is therefore assigned $\operatorname{nnz}(B)/\operatorname{nnz}(A_0)$ work units.
For a PCG solve with $k$ iterations, define
\begin{equation}\label{eq:num_wpd}
  d_k=\log_{10}\!\left(\frac{\|r_0\|_2}{\|r_k\|_2}\right),
  \qquad
  \mathrm{WPD}=\frac{k\bigl(w_{\mathrm{amg}}+w_{\mathrm{pcg}}\bigr)}{d_k},
\end{equation}
where $r_k=\Xb-A\Xx_k$, $w_{\mathrm{amg}}$ is the estimated cost of one application of an AMG preconditioning V-cycle.
and $w_{\mathrm{pcg}}$ is the outer PCG cost.  
All hierarchy levels are included. The cost estimates include relaxation, residual products, transfers, coarse grid corrections, and coarse solves.  
Setup costs are excluded. A Gauss--Seidel sweep is approximated by $\operatorname{nnz}(B)/\operatorname{nnz}(A_0)$. 
We estimate $ w_{\mathrm{pcg}} = 1$ for all methods. 

\paragraph{Ned, Ref, and Alg.}
Set $a_l=\operatorname{nnz}(A_l)$ and $p_l=\operatorname{nnz}(P_l)$. 
Let $I_{l,j}$ be the set of edge unknowns in patch $j$ on level $l$. 
Define
\[
 q_l=\sum_j\sum_{i\in I_{l,j}}\operatorname{nnz}(A_l(:,i)),
 \qquad s_l=\sum_j |I_{l,j}|^2,
\]
where $q_l$ counts incremental residual updates in one ordered Schwarz sweep,
and $s_l$ counts the local dense patch solves. The complete OSS-$\ell_1$
$V(1,1)$ cycle costs
\[
 w_{\mathrm{amg}}^{\mathrm{alg}}
 = w_{\mathrm{amg}}^{\mathrm{Ned}}
 = w_{\mathrm{amg}}^{\mathrm{ref}}
 =\frac{1}{a_0}\left[
   \sum_{l=0}^{L-2}
   (5a_l+2q_l+2s_l+2p_l)
   +2\operatorname{nnz}(H_c)
   -\dim(A_{L-1})\right],
\]
where $A_{L-1}=H_cH_c^T$ is the Cholesky factorization of the coarsest operator.
At level $\ell$, $5a_l$ counts the residual updates for two Jacobi, two Schwarz, and the coarse-grid correction.
The remaining terms count two Schwarz sweeps, two transfers, and two Cholesky coarse triangular solves. 
The final $\dim(A_{L-1})$ accounts for diagonal entries. 

\paragraph{AMS.}
AMS uses the default hypre 2.32.0 configuration with cycle type 13. 
In 2D, one cycle includes edge relaxation, vector nodal $x$-coordinate, vector nodal $y$-coordinate, and gradient corrections, 
followed by $y$ and $x$ corrections and edge relaxation. 
The residual update counts $5a_0$ for the five auxiliary corrections.
Each edge relaxation is one symmetric Gauss--Seidel sweep, and hence counts $4a_0$.
Each auxiliary correction forms an edge residual, restricts it,
applies one BoomerAMG V-cycle, and interpolates it. 
Let $p_G = \operatorname{nnz}(P_G)$, $p_x = \operatorname{nnz}(P_x)$, $p_y = \operatorname{nnz}(P_y)$.
Then
\[
 w_{\mathrm{amg}}^{\mathrm{AMS}}
 =\frac{1}{a_0}(9a_0+2p_G+4p_x+4p_y+v_G+2v_x+2v_y),
\]
where, for $z\in\{G,x,y\}$,
\[
 v_z=\sum_{l=0}^{L_z-1}
       \left(5\operatorname{nnz}(A_{z,l})
             +2\operatorname{nnz}(P_{z,l})\right)
       +2\operatorname{nnz}(A_{z,L_z})
\]
estimates counts for one BoomerAMG V-cycle for the auxiliary space.
Here, 5 counts pre- and post- symmetric Gauss-Seidel and coarse-grid correction residual,
2 counts the two transfers, and the last 2 counts the coarest grid symmetric Gauss-Seidel. 

% \paragraph{MueLu.}
% The Maxwell1 configuration applies one edge-hierarchy V-cycle.
% Each noncoarse level uses pre- and post-Hiptmair smoothing. 
% Each Hiptmair smoothing consists of one edge symmetric Gauss--Seidel, 
% one nodal symmetric Gauss--Seidel, and one edge symmetric Gauss--Seidel.
% Let $b_l=\operatorname{nnz}(B_l)$ for the nodal smoothing operator and
% $g_l=\operatorname{nnz}(G_l)$. Then
% \[
%  w_{\mathrm{amg}}^{\mathrm{MueLu}}
%  =\frac{1}{a_0}\left[
%     \sum_{l=0}^{L-1}15a_l+4b_l+4g_l+2p_l+(\dim(A_{L-1}))^2\right],
% \]
% where $15a_l$ counts eight edge Gauss-Seidel and six Hiptmair residuals, 
% $4b_l$ counts the two nodal symmetric Gauss-Seidel, $4g_l$ counts the nodal gradient transfers,
% $2p_l$ counts the edge transfers, and finally the KLU coarsest solve estimate.

\subsubsection{Interpolation and smoothing diagnostics}
\label{sec:numerical_stability_diagnostics}
For interpolation $P$ and restriction $R$ satisfying $RP=I_c$, we estimate
the sharp energy-stability bound
\begin{equation}\label{eq:numerical_eta_hat}
  \widehat{\eta}_P :=\sup_{\Xe\ne0}
    \frac{\|(I-PR)\Xe\|_A^2}{\|\Xe\|_A^2}.
\end{equation}
For $P=P_{\mathrm{app}}$, this is the smallest admissible constant in~\eqref{eq:Papp_energy_stability_IQ}. 
We note, however, it is distinct from the analytical uniform upper bound $\eta_{\mathrm{app}}=\max\{1,C_T^2\}$. 
Because the quotient depends on both $A$ and $P$,
$\widehat{\eta}_P$ may vary with $\mu$ while remaining below the same $\mu$-independent bound $\eta_{\mathrm{app}}$. 
Tracking $\widehat{\eta}_P$ shows whether the observed energy stability of the interpolation deteriorates as the coefficient in $\mu$ changes.
For Ref and Alg, $R$ is naturally defined. For Ned, since its edge orientation is the same as Ref but with different scaling,
we use the $R_{\mathrm{Ned}}=\sqrt{2}R_{\mathrm{ref}}$ so that $R_{\mathrm{Ned}}P_{\mathrm{Ned}}=I_c$.

To isolate the convergence effect from the smoother, 
we estimate a sharp bound of the smoothing assumption~\eqref{eq:uniform_fine_smoothing}
\begin{equation}\label{eq:numerical_csm_hat}
 \widehat{C}_{\mathrm{sm}}:=\sup_{\substack{\Xz\ne0\\R\Xz=0}}
       \frac{\Xz^T\widetilde M\Xz}{\Xz^TA\Xz},
 \qquad
 \widetilde M=M(M+M^T-A)^{-1}M^T.
\end{equation}
A large and increasing $\widehat{C}_{\mathrm{sm}}$ indicates that the smoother loses robustness as the coefficient contrast grows.

We report these two estimates for selected test examples.
All reported quantities are numerical estimates of the corresponding suprema, not certified upper bounds.

% uniform triangle
\subsection{Isotropy on 2D uniform triangular meshes}

Table~\ref{tab:uniform} establishes the isotropic baseline on uniformly refined triangular meshes.  
It provides the reference for the jump experiments below.

\begin{table}[!htbp]
\centering
\caption{Multilevel results of isotropic uniform triangular-mesh.  Panels report
(a) PCG iterations, (b) $N_{\mathrm{level}}/\mathrm{OC}$,
(c) maximum commuting defects $\delta_{\max}$, and (d) estimated work per digit (WPD)}
\label{tab:uniform}
\textit{(a) PCG iterations}\par\smallskip

%% with Muelu
% \begin{tabular}{rrrrrrr}
% \toprule
% $L$ & $n$ & Ned & Ref & Alg & AMS & MueLu \\
% \midrule
% \csvreader[separator=semicolon,late after line={\\\paperafterrow}]
%   {numerics/paper_tables/uniform_iterations.csv}
%   {level=\papercolA,unknowns=\papercolB,geo=\papercolC,ref=\papercolD,alg=\papercolE,ams=\papercolF,muelu=\papercolG,afterrow=\paperafterrow}
%   {\papercolA & \papercolB & \papercolC & \papercolD & \papercolE & \papercolF & \papercolG}
% \bottomrule
% \end{tabular}

%% without Muelu
\begin{tabular}{rrrrrr}
\toprule
$L$ & $n$ & Ned & Ref & Alg & AMS \\
\midrule
\begin{filecontents*}{numerics/paper_tables/uniform_iterations.csv}
level;unknowns;geo;ref;alg;ams;muelu;afterrow
3;6,816;\textbf{9};\textbf{9};10;13;17;
4;27,456;\textbf{9};\textbf{9};\textbf{9};14;17;
5;110,208;\textbf{9};\textbf{9};\textbf{9};14;17;
6;441,600;\textbf{9};\textbf{9};\textbf{9};15;19;
\end{filecontents*}
\csvreader[separator=semicolon,late after line={\\\paperafterrow}]
  {numerics/paper_tables/uniform_iterations.csv}
  {level=\papercolA,unknowns=\papercolB,geo=\papercolC,ref=\papercolD,alg=\papercolE,ams=\papercolF,afterrow=\paperafterrow}
  {\papercolA & \papercolB & \papercolC & \papercolD & \papercolE & \papercolF}
\bottomrule
\end{tabular}

\par\medskip
\textit{(b) $N_{\mathrm{level}}/\mathrm{OC}$}\par\smallskip

%% with Muelu
% \begin{tabular}{rrrrrr}
% \toprule
% $L$ & $n$ & Ned & Ref & Alg & MueLu \\
% \midrule
% \csvreader[separator=semicolon,late after line={\\\paperafterrow}]
%   {numerics/paper_tables/uniform_complexity.csv}
%   {level=\papercolA,unknowns=\papercolB,geo=\papercolC,ref=\papercolD,alg=\papercolE,muelu=\papercolF,afterrow=\paperafterrow}
%   {\papercolA & \papercolB & \papercolC & \papercolD & \papercolE & \papercolF}
% \bottomrule
% \end{tabular}

%% without Muelu
\begin{tabular}{rrrrr}
\toprule
$L$ & $n$ & Ned & Ref & Alg \\
\midrule
\begin{filecontents*}{numerics/paper_tables/uniform_complexity.csv}
level;unknowns;geo;ref;alg;muelu;afterrow
3;6,816;$4/1.314$;$4/1.314$;$4/1.314$;$3/1.150$;
4;27,456;$5/1.325$;$5/1.325$;$5/1.324$;$4/1.110$;
5;110,208;$6/1.329$;$6/1.329$;$6/1.329$;$4/1.090$;
6;441,600;$7/1.331$;$7/1.331$;$7/1.331$;$5/1.080$;
\end{filecontents*}
\csvreader[separator=semicolon,late after line={\\\paperafterrow}]
  {numerics/paper_tables/uniform_complexity.csv}
  {level=\papercolA,unknowns=\papercolB,geo=\papercolC,ref=\papercolD,alg=\papercolE,afterrow=\paperafterrow}
  {\papercolA & \papercolB & \papercolC & \papercolD & \papercolE}
\bottomrule
\end{tabular}

\par\medskip
\textit{(c) $\delta_{\max}$}\par\smallskip
\begin{tabular}{rrrr}
\toprule
$L$ & $n$ & Ref & Alg \\
\midrule
\begin{filecontents*}{numerics/paper_tables/uniform_commuting.csv}
level;unknowns;ref;alg;afterrow
3;6,816;$1.07\times10^{-5}$;$1.34\times10^{-5}$;
4;27,456;$1.07\times10^{-5}$;$1.34\times10^{-5}$;
5;110,208;$1.07\times10^{-5}$;$1.34\times10^{-5}$;
6;441,600;$1.07\times10^{-5}$;$1.34\times10^{-5}$;
\end{filecontents*}
\csvreader[separator=semicolon,late after line={\\\paperafterrow}]
  {numerics/paper_tables/uniform_commuting.csv}
  {level=\papercolA,unknowns=\papercolB,ref=\papercolC,alg=\papercolD,afterrow=\paperafterrow}
  {\papercolA & \papercolB & \papercolC & \papercolD}
\bottomrule
\end{tabular}
\par\medskip
\textit{(d) WPD}\par\smallskip

%% with Muelu
% \begin{tabular}{rrrrrrr}
% \toprule
% $L$ & $n$ & Ned & Ref & Alg & AMS & MueLu \\
% \midrule
% \csvreader[separator=semicolon,late after line={\\\paperafterrow}]
%   {numerics/paper_tables/uniform_wpd.csv}
%   {level=\papercolA,unknowns=\papercolB,geo=\papercolC,ref=\papercolD,alg=\papercolE,ams=\papercolF,muelu=\papercolG,afterrow=\paperafterrow}
%   {\papercolA & \papercolB & \papercolC & \papercolD & \papercolE & \papercolF & \papercolG}
% \bottomrule
% \end{tabular}

%% without Muelu
\begin{tabular}{rrrrrr}
\toprule
$L$ & $n$ & Ned & Ref & Alg & AMS \\
\midrule
\begin{filecontents*}{numerics/paper_tables/uniform_wpd.csv}
level;unknowns;geo;ref;alg;ams;muelu;afterrow
3;6,816;\textbf{26.33};26.45;28.80;51.61;46.39;
4;27,456;25.12;\textbf{24.98};27.27;54.36;45.02;
5;110,208;23.27;\textbf{22.99};28.78;57.83;46.94;
6;441,600;\textbf{25.54};25.55;27.22;64.88;53.73;
\end{filecontents*}
\csvreader[separator=semicolon,late after line={\\\paperafterrow}]
  {numerics/paper_tables/uniform_wpd.csv}
  {level=\papercolA,unknowns=\papercolB,geo=\papercolC,ref=\papercolD,alg=\papercolE,ams=\papercolF,afterrow=\paperafterrow}
  {\papercolA & \papercolB & \papercolC & \papercolD & \papercolE & \papercolF}
\bottomrule
\end{tabular}

\end{table}

The isotropic test shows the proposed interpolations have little performance 
difference from N\'ed\'elec interpolation. In this case, Ref reproduces Ned, while
Alg remains mesh independent with similar operator complexity. The maximum commuting
defects for both Ref and Alg remains constant at $O(10^{-5})$, which suggests its independence
to mesh resolution. Panel (d) shows that this similarity also holds in estimated solve work:
Ned, Ref, and Alg have comparable, nearly mesh-independent WPD. This test
provides a clean baseline for the next jump experiments.

\subsection{Nonaligned rotated checkerboard jumps on 2D uniform triangular meshes}\label{subsec:uniform_jump} 
Having established mesh-independent behavior in the isotropic setting, we
next test robustness under strong, nonaligned coefficient variation.  The
base mesh partitions $\Omega=(0,1)^2$ uniformly into $6\times6$ squares,
with each square split into two triangles; subsequent meshes are obtained by
uniform refinement. Independently of this mesh hierarchy, we impose a fixed
rotated $8\times8$ checkerboard coefficient pattern. An illustration is shown in
Figure~\ref{fig:rotated_checkerboard}. For
$\mathbf{x}=(x,y)^T$, let $\mathbf{x}_c=(1/2,1/2)^T$ be the mesh center and define the rotated,
scaled, and shifted coordinates $\widehat{\mathbf{x}}=(\widehat{x},\widehat{y})^T$ by
\[
 \widehat{\mathbf{x}}
 =8 \begin{pmatrix}
   \cos 22.5^\circ & \sin 22.5^\circ\\
  -\sin 22.5^\circ & \cos 22.5^\circ
  \end{pmatrix}
 (\mathbf{x}-\mathbf{x}_c)
  +\begin{pmatrix}0.17\\0.31\end{pmatrix}.
\]
For $\mu_{\min}\in\{10^{-2},10^{-4},10^{-6}\}$, the coefficient is
$\mu=\mu_{\min}$ where
$\lfloor\widehat{x}\rfloor+\lfloor\widehat{y}\rfloor$ is even and
$\mu=1$ otherwise. 
Each element receives the coefficient evaluated at its
barycenter. The rotation and shift are chosen so that the jump interfaces remain nonaligned with the mesh under refinement, 
which prevents the geometric refinement hierarchy from inadvertently resolving the coefficient discontinuities. The jump remains fixed under refinement while its
nonaligned interfaces are more refined as $L$ increases.

\begin{table}[!htbp]
\centering
\caption{Multilevel results of rotated-checkerboard jump on uniformly refined triangular
meshes.  Panels report (a) PCG iterations, (b) $N_{\mathrm{level}}/\mathrm{OC}$,
(c) maximum commuting defects $\delta_{\max}$, and (d) estimated work per digit (WPD)}
\label{tab:rotated_checkerboard}
\small
\setlength{\tabcolsep}{3.5pt}
\textit{(a) PCG iterations}\par\smallskip

%% with Muelu
% \begin{tabular}{crrrrrrr}
% \toprule
% $\mu_{\min}$ & $L$ & $n$ & Ned & Ref & Alg & AMS & MueLu \\
% \midrule
% \csvreader[separator=semicolon,late after line={\\\paperafterrow}]
%   {numerics/paper_tables/checkerboard_iterations.csv}
%   {mu=\papercolA,level=\papercolB,unknowns=\papercolC,geo=\papercolD,ref=\papercolE,alg=\papercolF,ams=\papercolG,muelu=\papercolH,afterrow=\paperafterrow}
%   {\papercolA & \papercolB & \papercolC & \papercolD & \papercolE & \papercolF & \papercolG & \papercolH}
% \bottomrule
% \end{tabular}

%% without Muelu
\begin{tabular}{crrrrrr}
\toprule
$\mu_{\min}$ & $L$ & $n$ & Ned & Ref & Alg & AMS \\
\midrule
\begin{filecontents*}{numerics/paper_tables/checkerboard_iterations.csv}
mu;level;unknowns;geo;ref;alg;ams;muelu;afterrow
$10^{-2}$;3;6,816;38;14;\textbf{13};15;42;
;4;27,456;39;14;\textbf{13};16;40;
;5;110,208;40;17;\textbf{14};18;47;
;6;441,600;41;\textbf{14};\textbf{14};\textbf{14};38;
$10^{-4}$;3;6,816;129;\textbf{14};\textbf{14};18;204;\addlinespace[2pt]
;4;27,456;168;14;\textbf{13};19;218;
;5;110,208;200;18;\textbf{14};25;242;
;6;441,600;173;16;\textbf{14};15;190;
$10^{-6}$;3;6,816;117;15;\textbf{13};26;425;\addlinespace[2pt]
;4;27,456;149;14;\textbf{13};25;522;
;5;110,208;188;17;\textbf{14};37;780;
;6;441,600;261;17;\textbf{15};17;612;
\end{filecontents*}
\csvreader[separator=semicolon,late after line={\\\paperafterrow}]
  {numerics/paper_tables/checkerboard_iterations.csv}
  {mu=\papercolA,level=\papercolB,unknowns=\papercolC,geo=\papercolD,ref=\papercolE,alg=\papercolF,ams=\papercolG,afterrow=\paperafterrow}
  {\papercolA & \papercolB & \papercolC & \papercolD & \papercolE & \papercolF & \papercolG}
\bottomrule
\end{tabular}

\par\medskip
\textit{(b) $N_{\mathrm{level}}/\mathrm{OC}$}\par\smallskip

%% with Muelu
% \begin{tabular}{rrrrrr}
% \toprule
% $L$ & $n$ & Ned & Ref & Alg & MueLu \\
% \midrule
% \csvreader[separator=semicolon,late after line={\\\paperafterrow}]
%   {numerics/paper_tables/checkerboard_complexity.csv}
%   {level=\papercolA,unknowns=\papercolB,geo=\papercolC,ref=\papercolD,alg=\papercolE,muelu=\papercolF,afterrow=\paperafterrow}
%   {\papercolA & \papercolB & \papercolC & \papercolD & \papercolE & \papercolF}
% \bottomrule
% \end{tabular}

%% without Muelu
\begin{tabular}{rrrrr}
\toprule
$L$ & $n$ & Ned & Ref & Alg \\
\midrule
\begin{filecontents*}{numerics/paper_tables/checkerboard_complexity.csv}
level;unknowns;geo;ref;alg;muelu;afterrow
3;6,816;$4/1.314$;$4/1.314$;$4/1.314$;$3/1.150$;
4;27,456;$5/1.325$;$5/1.325$;$5/1.324$;$4/1.110$;
5;110,208;$6/1.329$;$6/1.329$;$6/1.329$;$4/1.090$;
6;441,600;$7/1.331$;$7/1.331$;$7/1.331$;$5/1.080$;
\end{filecontents*}
\csvreader[separator=semicolon,late after line={\\\paperafterrow}]
  {numerics/paper_tables/checkerboard_complexity.csv}
  {level=\papercolA,unknowns=\papercolB,geo=\papercolC,ref=\papercolD,alg=\papercolE,afterrow=\paperafterrow}
  {\papercolA & \papercolB & \papercolC & \papercolD & \papercolE}
\bottomrule
\end{tabular}

\par\medskip
\textit{(c) $\delta_{\max}$}\par\smallskip
\begin{tabular}{rrrrrrr}
\toprule
& \multicolumn{2}{c}{$\mu_{\min}=10^{-2}$}
& \multicolumn{2}{c}{$\mu_{\min}=10^{-4}$}
& \multicolumn{2}{c}{$\mu_{\min}=10^{-6}$} \\
\cmidrule(lr){2-3}\cmidrule(lr){4-5}\cmidrule(lr){6-7}
$L$ & Ref & Alg & Ref & Alg & Ref & Alg \\
\midrule
\begin{filecontents*}{numerics/paper_tables/checkerboard_commuting.csv}
level;ref_low;alg_low;ref_medium;alg_medium;ref_high;alg_high;afterrow
3;$5.77\times10^{-4}$;$6.73\times10^{-4}$;$0.0486$;$0.0542$;$0.362$;$0.365$;
4;$5.73\times10^{-4}$;$6.67\times10^{-4}$;$0.0482$;$0.0538$;$0.360$;$0.364$;
5;$5.71\times10^{-4}$;$6.65\times10^{-4}$;$0.0481$;$0.0537$;$0.362$;$0.365$;
6;$5.71\times10^{-4}$;$6.65\times10^{-4}$;$0.0481$;$0.0537$;$0.362$;$0.365$;
\end{filecontents*}
\csvreader[separator=semicolon,late after line={\\\paperafterrow}]
  {numerics/paper_tables/checkerboard_commuting.csv}
  {level=\papercolA,ref_low=\papercolB,alg_low=\papercolC,ref_medium=\papercolD,alg_medium=\papercolE,ref_high=\papercolF,alg_high=\papercolG,afterrow=\paperafterrow}
  {\papercolA & \papercolB & \papercolC & \papercolD & \papercolE & \papercolF & \papercolG}
\bottomrule
\end{tabular}
\par\medskip
\textit{(d) WPD}\par\smallskip

%% with Muelu
% \begin{tabular}{crrrrrrr}
% \toprule
% $\mu_{\min}$ & $L$ & $n$ & Ned & Ref & Alg & AMS & MueLu \\
% \midrule
% \csvreader[separator=semicolon,late after line={\\\paperafterrow}]
%   {numerics/paper_tables/checkerboard_wpd.csv}
%   {mu=\papercolA,level=\papercolB,unknowns=\papercolC,geo=\papercolD,ref=\papercolE,alg=\papercolF,ams=\papercolG,muelu=\papercolH,afterrow=\paperafterrow}
%   {\papercolA & \papercolB & \papercolC & \papercolD & \papercolE & \papercolF & \papercolG & \papercolH}
% \bottomrule
% \end{tabular}

%% without Muelu
\begin{tabular}{crrrrrr}
\toprule
$\mu_{\min}$ & $L$ & $n$ & Ned & Ref & Alg & AMS \\
\midrule
\begin{filecontents*}{numerics/paper_tables/checkerboard_wpd.csv}
mu;level;unknowns;geo;ref;alg;ams;muelu;afterrow
$10^{-2}$;3;6,816;93.71;31.85;\textbf{31.31};58.38;117.99;
;4;27,456;97.91;32.14;\textbf{31.85};60.32;111.20;
;5;110,208;101.90;40.52;\textbf{33.33};72.33;127.54;
;6;441,600;108.36;\textbf{35.83};\textbf{35.83};59.36;107.39;
$10^{-4}$;3;6,816;305.31;33.41;\textbf{30.48};68.30;586.15;\addlinespace[2pt]
;4;27,456;422.40;33.83;\textbf{31.35};73.08;611.19;
;5;110,208;515.79;42.91;\textbf{33.52};101.75;675.88;
;6;441,600;474.76;41.21;\textbf{36.02};62.96;559.63;
$10^{-6}$;3;6,816;279.76;31.95;\textbf{31.89};91.59;1215.41;\addlinespace[2pt]
;4;27,456;360.66;33.26;\textbf{31.60};95.46;1433.53;
;5;110,208;485.38;43.09;\textbf{33.59};144.60;2231.02;
;6;441,600;723.89;43.85;\textbf{38.44};71.60;1864.19;
\end{filecontents*}
\csvreader[separator=semicolon,late after line={\\\paperafterrow}]
  {numerics/paper_tables/checkerboard_wpd.csv}
  {mu=\papercolA,level=\papercolB,unknowns=\papercolC,geo=\papercolD,ref=\papercolE,alg=\papercolF,ams=\papercolG,afterrow=\paperafterrow}
  {\papercolA & \papercolB & \papercolC & \papercolD & \papercolE & \papercolF & \papercolG}
\bottomrule
\end{tabular}

\end{table}

\begin{figure}[!htbp]
  \centering
  \includegraphics[width=0.42\textwidth]{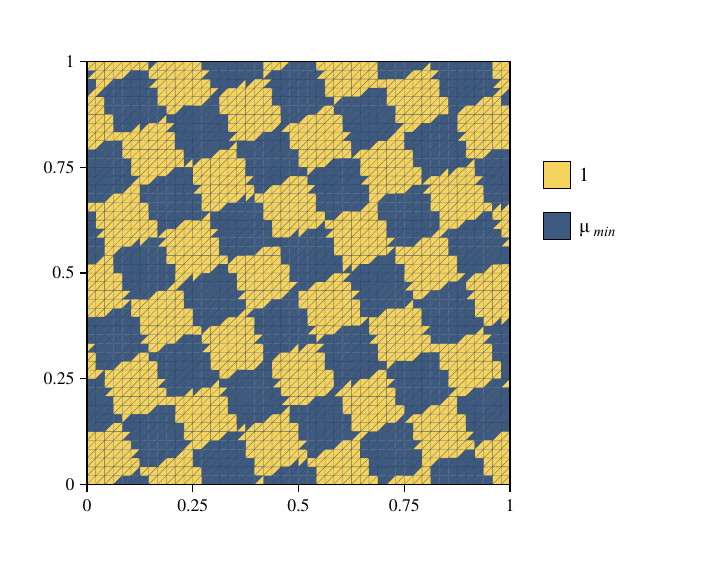}
  \caption{The rotated $8\times8$ checkerboard jump coefficients.}
  \label{fig:rotated_checkerboard}
\end{figure}

Table~\ref{tab:rotated_checkerboard} provides a controlled comparison of Ned, Ref, and Alg,
as they use the same overlapping-Schwarz smoother and have comparable numbers of levels and operator complexities.
Ned is highly sensitive to the coefficient jumps, whereas Ref and Alg remain stable as both $L$ and the contrast increase. 
This demonstrates robustness to finer resolution of the fixed interfaces and to jump magnitude. 
AMS is more sensitive to this geometry. 
Its iteration counts vary less regularly because its performance depends on the auxiliary nodal hierarchies and their coupling to the edge problem.
The WPD results in panel (d) preserve this advantage after accounting for
the different cycle costs: Ref and Alg require substantially less
estimated WPD than Ned and AMS across the jump tests,
and remain nearly refinement independent.

The commuting results add a complementary observation.  
For each degree of contrast, $\delta_{\max}$ shows refinement-independence.
It does increase with contrast, but Ref and Alg retain nearly constant PCG performance even when the commuting defect is no longer small.  
Consequently, this shows that $\delta_{\max}$ is a useful structural diagnostic, 
but it is not by itself a convergence predictor.

\subsection{Nonaligned rotated checkerboard jump on an unstructured Delaunay mesh}
Finally, we repeat the rotated-checkerboard jump experiment on an unstructured mesh. 
We fix $\mu_{\min}=10^{-6}$. The coarsest grid is the Delaunay
triangulation in Figure~\ref{fig:delaunay_fig}(a), with 68 triangles and 92 interior edge unknowns.  
Each subsequent mesh is obtained by uniform refinement of this triangulation.  
On every fine problem, the same physical checkerboard from Section~\ref{subsec:uniform_jump} is evaluated at element barycenters.  
The $L=3$ discretization of this coefficient is shown in
Figure~\ref{fig:delaunay_fig}(b). 

\begin{figure}[!htbp]
    \centering
    \begin{subfigure}[t]{0.32\textwidth}
        \centering
        \includegraphics[width=\linewidth]{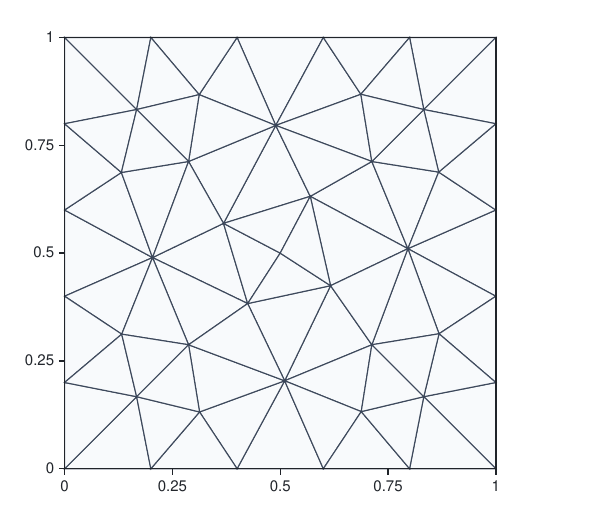}
        \caption{Coarsest mesh}
        \label{fig:delaunay_coarse_mesh}
    \end{subfigure}
    \hspace{0.035\textwidth}
    \begin{subfigure}[t]{0.32\textwidth}
        \centering
        \includegraphics[width=\linewidth]{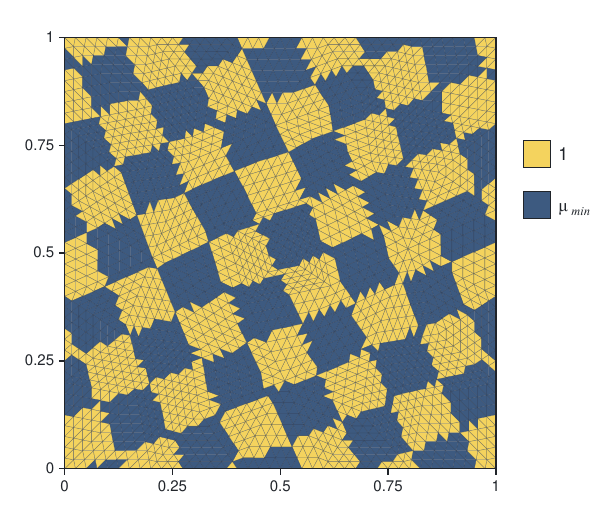}
        \caption{Level-$3$ jump discretization}
        \label{fig:delaunay_jump}
    \end{subfigure}
    \caption{Unstructured Delaunay rotated-checkerboard test.  Left: coarsest mesh.  Right: the rotated $8\times8$ checkerboard
    sampled on the level-$3$ mesh.}
    \label{fig:delaunay_fig}
\end{figure}

\begin{table}[!htbp]
\centering
\caption{Multilevel results of rotated-checkerboard jump with $\mu_{\min}=10^{-6}$ on
uniformly refined unstructured Delaunay meshes.  Panels report (a) PCG
iterations, (b) $N_{\mathrm{level}}/\mathrm{OC}$, (c) maximum commuting defects $\delta_{\max}$, and (d) estimated work per digit (WPD)}
\label{tab:delaunay}
\textit{(a) PCG iterations}\par\smallskip

%% with Muelu
% \begin{tabular}{rrrrrrr}
% \toprule
% $L$ & $n$ & Ned & Ref & Alg & AMS & MueLu \\
% \midrule
% \csvreader[separator=semicolon,late after line={\\\paperafterrow}]
%   {numerics/paper_tables/delaunay_iterations.csv}
%   {level=\papercolA,unknowns=\papercolB,geo=\papercolC,ref=\papercolD,alg=\papercolE,ams=\papercolF,muelu=\papercolG,afterrow=\paperafterrow}
%   {\papercolA & \papercolB & \papercolC & \papercolD & \papercolE & \papercolF & \papercolG}
% \bottomrule
% \end{tabular}

%% without Muelu
\begin{tabular}{rrrrrr}
\toprule
$L$ & $n$ & Ned & Ref & Alg & AMS \\
\midrule
\begin{filecontents*}{numerics/paper_tables/delaunay_iterations.csv}
level;unknowns;geo;ref;alg;ams;muelu;afterrow
3;6,448;49;\textbf{11};12;213;334;
4;25,952;82;\textbf{9};14;301;475;
5;104,128;116;\textbf{9};14;282;461;
6;417,152;211;\textbf{11};14;317;576;
\end{filecontents*}
\csvreader[separator=semicolon,late after line={\\\paperafterrow}]
  {numerics/paper_tables/delaunay_iterations.csv}
  {level=\papercolA,unknowns=\papercolB,geo=\papercolC,ref=\papercolD,alg=\papercolE,ams=\papercolF,afterrow=\paperafterrow}
  {\papercolA & \papercolB & \papercolC & \papercolD & \papercolE & \papercolF}
\bottomrule
\end{tabular}

\par\medskip
\textit{(b) $N_{\mathrm{level}}/\mathrm{OC}$}\par\smallskip

%% with Muelu
% \begin{tabular}{rrrrrr}
% \toprule
% $L$ & $n$ & Ned & Ref & Alg & MueLu \\
% \midrule
% \csvreader[separator=semicolon,late after line={\\\paperafterrow}]
%   {numerics/paper_tables/delaunay_complexity.csv}
%   {level=\papercolA,unknowns=\papercolB,geo=\papercolC,ref=\papercolD,alg=\papercolE,muelu=\papercolF,afterrow=\paperafterrow}
%   {\papercolA & \papercolB & \papercolC & \papercolD & \papercolE & \papercolF}
% \bottomrule
% \end{tabular}

%% without Muelu
\begin{tabular}{rrrrr}
\toprule
$L$ & $n$ & Ned & Ref & Alg \\
\midrule
\begin{filecontents*}{numerics/paper_tables/delaunay_complexity.csv}
level;unknowns;geo;ref;alg;muelu;afterrow
3;6,448;$4/1.316$;$4/1.316$;$4/1.385$;$3/1.140$;
4;25,952;$5/1.325$;$5/1.325$;$5/1.496$;$4/1.100$;
5;104,128;$6/1.330$;$6/1.330$;$6/1.541$;$4/1.080$;
6;417,152;$7/1.332$;$7/1.332$;$7/1.884$;$5/1.080$;
\end{filecontents*}
\csvreader[separator=semicolon,late after line={\\\paperafterrow}]
  {numerics/paper_tables/delaunay_complexity.csv}
  {level=\papercolA,unknowns=\papercolB,geo=\papercolC,ref=\papercolD,alg=\papercolE,afterrow=\paperafterrow}
  {\papercolA & \papercolB & \papercolC & \papercolD & \papercolE}
\bottomrule
\end{tabular}

\par\medskip
\textit{(c) $\delta_{\max}$}\par\smallskip
\begin{tabular}{rrrr}
\toprule
$L$ & $n$ & Ref & Alg \\
\midrule
\begin{filecontents*}{numerics/paper_tables/delaunay_commuting.csv}
level;unknowns;ref;alg;afterrow
3;6,448;$0.348$;$0.457$;
4;25,952;$0.345$;$0.358$;
5;104,128;$0.345$;$0.298$;
6;417,152;$0.345$;$0.208$;
\end{filecontents*}
\csvreader[separator=semicolon,late after line={\\\paperafterrow}]
  {numerics/paper_tables/delaunay_commuting.csv}
  {level=\papercolA,unknowns=\papercolB,ref=\papercolC,alg=\papercolD,afterrow=\paperafterrow}
  {\papercolA & \papercolB & \papercolC & \papercolD}
\bottomrule
\end{tabular}
\par\medskip
\textit{(d) WPD}\par\smallskip

%% with Muelu
% \begin{tabular}{rrrrrrr}
% \toprule
% $L$ & $n$ & Ned & Ref & Alg & AMS & MueLu \\
% \midrule
% \csvreader[separator=semicolon,late after line={\\\paperafterrow}]
%   {numerics/paper_tables/delaunay_wpd.csv}
%   {level=\papercolA,unknowns=\papercolB,geo=\papercolC,ref=\papercolD,alg=\papercolE,ams=\papercolF,muelu=\papercolG,afterrow=\paperafterrow}
%   {\papercolA & \papercolB & \papercolC & \papercolD & \papercolE & \papercolF & \papercolG}
% \bottomrule
% \end{tabular}

%% without Muelu
\begin{tabular}{rrrrrr}
\toprule
$L$ & $n$ & Ned & Ref & Alg & AMS \\
\midrule
\begin{filecontents*}{numerics/paper_tables/delaunay_wpd.csv}
level;unknowns;geo;ref;alg;ams;muelu;afterrow
3;6,448;119.92;\textbf{24.54};28.50;761.67;934.66;
4;25,952;199.24;\textbf{21.97};36.36;1131.38;1290.68;
5;104,128;295.17;\textbf{21.48};38.98;1116.27;1300.37;
6;417,152;582.35;\textbf{28.54};47.70;1377.79;1752.33;
\end{filecontents*}
\csvreader[separator=semicolon,late after line={\\\paperafterrow}]
  {numerics/paper_tables/delaunay_wpd.csv}
  {level=\papercolA,unknowns=\papercolB,geo=\papercolC,ref=\papercolD,alg=\papercolE,ams=\papercolF,afterrow=\paperafterrow}
  {\papercolA & \papercolB & \papercolC & \papercolD & \papercolE & \papercolF}
\bottomrule
\end{tabular}

\end{table}

Table~\ref{tab:delaunay} demonstrates robustness with respect to mesh geometry aside from robustness with respect to contrast. 
Ref retains its mesh-independent behavior and its operator complexity continues to track Ned.
By contrast, the deterioration of Ned highlights Ref's robustness to mesh misalignment.
The fact that AMS struggles confirms that the difficulty observed above in Table~\ref{tab:rotated_checkerboard} persists on an unstructured mesh.
Moreover, the irregular mesh geometry substantially increases both the PCG iteration counts and WPD of Ned and AMS,
whereas Ref and Alg retain much lower, nearly mesh-independent iteration counts. 
Ref also maintains nearly constant WPD, while Alg’s WPD grows moderately with operator complexity.

Alg presents a different tradeoff.  Its iteration counts remain stable, 
so the algebraically selected coarse grid continues to capture the relevant low-energy mode without using geometric hierarchy information.
But its operator complexity grows under refinement $L$. 
This indicates fill-in on coarse levels as a limitation in the unstructured hierarchy, 
while the resulting interpolation remains highly effective.
% MueLu exhibits the converse: its hierarchy remains sparse, but its large iteration counts causes high growing WPD, 
% which shows that low operator complexity alone does not guarantee an effective coarse correction.

\begin{table}[!htbp]
\centering
\caption{Estimates of $\widehat{\eta}_P$ for the two-dimensional
rotated-checkerboard problems.}
\label{tab:checkerboard2d_eta}
\begingroup
\small
\setlength{\tabcolsep}{5pt}
\textit{(a) Uniform triangular mesh:
$n_3=6,816$, $n_4=27,456$}\par\smallskip
\begin{tabular}{rrrrrrr}
\toprule
& \multicolumn{3}{c}{$L=3$} & \multicolumn{3}{c}{$L=4$} \\
\cmidrule(lr){2-4}\cmidrule(lr){5-7}
$\mu_{\min}$ & Ned & Ref & Alg & Ned & Ref & Alg \\
\midrule
\begin{filecontents*}{numerics/paper_tables/uniform_stability.csv}
mu;geo_level3;ref_level3;alg_level3;geo_level4;ref_level4;alg_level4;afterrow
$1$;1.996;1.996;1.996;1.999;1.999;1.999;
$10^{-2}$;25.50;1.996;1.996;25.50;1.999;1.999;
$10^{-4}$;$2.489\times10^3$;1.988;1.988;$2.499\times10^3$;1.997;1.997;
$10^{-6}$;$1.795\times10^5$;2.059;2.074;$2.331\times10^5$;1.964;1.964;
\end{filecontents*}
\csvreader[separator=semicolon,late after line={\\\paperafterrow}]
  {numerics/paper_tables/uniform_stability.csv}
  {mu=\papercolA,geo_level3=\papercolB,ref_level3=\papercolC,alg_level3=\papercolD,geo_level4=\papercolE,ref_level4=\papercolF,alg_level4=\papercolG,afterrow=\paperafterrow}
  {\papercolA & \papercolB & \papercolC & \papercolD & \papercolE & \papercolF & \papercolG}
\bottomrule
\end{tabular}
\par\medskip
\textit{(b) Delaunay mesh:
$n_3=6,448$, $n_4=25,952$}\par\smallskip
\begin{tabular}{rrrrrrr}
\toprule
& \multicolumn{3}{c}{$L=3$} & \multicolumn{3}{c}{$L=4$} \\
\cmidrule(lr){2-4}\cmidrule(lr){5-7}
$\mu_{\min}$ & Ned & Ref & Alg & Ned & Ref & Alg \\
\midrule
\begin{filecontents*}{numerics/paper_tables/delaunay_stability.csv}
mu;geo_level3;ref_level3;alg_level3;geo_level4;ref_level4;alg_level4;afterrow
$1$;1.965;1.965;2.237;2.035;2.035;2.363;
$10^{-2}$;25.50;1.965;2.237;25.50;2.035;2.363;
$10^{-4}$;$2.490\times10^3$;1.955;2.239;$2.498\times10^3$;2.031;2.350;
$10^{-6}$;$1.834\times10^5$;1.939;2.288;$2.303\times10^5$;1.970;2.275;
\end{filecontents*}
\csvreader[separator=semicolon,late after line={\\\paperafterrow}]
  {numerics/paper_tables/delaunay_stability.csv}
  {mu=\papercolA,geo_level3=\papercolB,ref_level3=\papercolC,alg_level3=\papercolD,geo_level4=\papercolE,ref_level4=\papercolF,alg_level4=\papercolG,afterrow=\paperafterrow}
  {\papercolA & \papercolB & \papercolC & \papercolD & \papercolE & \papercolF & \papercolG}
\bottomrule
\end{tabular}
\endgroup
\end{table}

For the two-dimensional checkerboard problems, we also estimate the interpolation quality $\widehat{\eta}_P$ at $L=3,4$ for Ned, Ref, Alg.
Table~\ref{tab:checkerboard2d_eta} shows that Ref and Alg retain small
energy-stability constants over the tested contrast range on both mesh families.
Ned instead grows from approximately $2$ in the isotropic case to $O(10^5)$ at contrast $10^6$.
The small variation in the Ref/Alg estimated constants is consistent with a
common contrast-independent upper bound $\eta_{\mathrm{app}}$in~\eqref{eq:Papp_energy_stability_IQ}.

\subsection{Three-dimensional proof of concept}\label{sec:num3d}
We present a three-dimensional experiment on an unstructured Delaunay tetrahedral hierarchy. The domain is the unit cube.
The refinement-based and complete-algebraic harmonic constructions are kept
strictly two-level because their three-dimensional interpolation operators
are too dense for practical multilevel use. 
We also include two-level N\'ed\'elec interpolation, while N\'ed\'elec multigrid, AMS, and ParELAG use their package configurations. 
% For MueLu, we set the nodal-hierarchy coarse-size threshold to 600 as the smaller default size triggers a PCG failure. 

The coarsest unstructured Delaunay mesh has 994 tetrahedra.  
The two reported refinements, labeled $L=1$ and $L=2$, have $n=7,690$ and $n=67,804$ interior edge unknowns, respectively.
We prescribe a $4\times4\times4$ checkerboard coefficient on the unit cube by defining
\[
j(\Xx)=\lfloor4x\rfloor+\lfloor4y\rfloor+\lfloor4z\rfloor .
\]
For $\mu_{\min}\in\{1,10^{-2},10^{-4},10^{-6}\}$, we set $\mu=\mu_{\min}$ when $j$ is even and $\mu=1$ when $j$ is odd, evaluating the coefficient at each tetrahedron barycenter.
The resulting 64 equal subcubes alternate between the two coefficient values.
Their physical interfaces remain nonaligned with the refinement hierarchies. 
Figure~\ref{fig:cube3d_checkerboard} shows the result jump configuration.
% We define an alternating slab configuration for jump distribution. 
% To define the alternating slabs, we introduce
% \[
%  s(\Xx)=\frac{x+2y+3z}{6},\qquad
%  j(\Xx)=\left\lfloor \frac{s(\Xx)+0.04}{1/11} \right\rfloor .
% \]
% where $s(\Xx)$ is the scalar projection of $\Xx$ onto the slab-normal direction $(1,2,3)$, 
% and $j(\Xx)$ maps to the index of the slab that contains $\Xx$.
% The spacing $1/11$ and the offset $0.04$ controls the slab width and the boundary interface offset.
% For $\mu_{\min}\in\{1,10^{-2},10^{-4},10^{-6}\}$, we set $\mu=1$ when $j$ is even and $\mu=\mu_{\min}$ when $j$ is odd, 
% evaluating the coefficient at each tetrahedron barycenter.  
% The resulting 12 slabs are fixed in physical space, and remain unchanged under refinement.  

\begin{figure}[!htbp]
  \centering
  % \includegraphics[width=0.52\textwidth]{pic/oblique_alternating_slabs_12bands_3d_delaunay_k2.pdf}
  % \caption{The nonaligned 12-band alternating-slab coefficient on the
  % $L=2$ Delaunay tetrahedral mesh. }
  % \label{fig:cube3d_slabs}
  \includegraphics[width=0.52\textwidth]{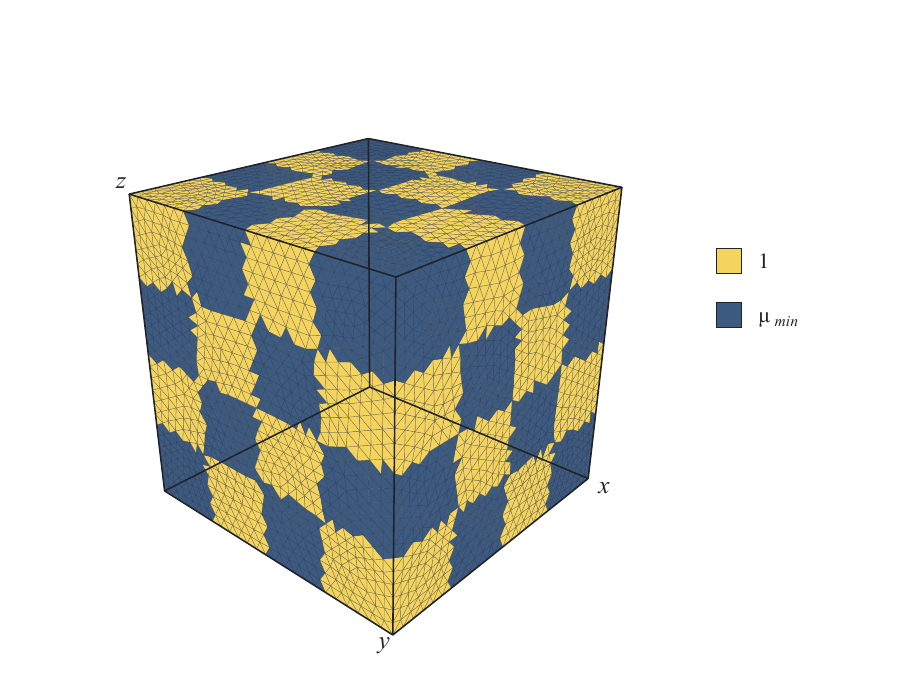}
  \caption{$4\times 4 \times 4$ checkerboard coefficient on the $L=2$ Delaunay tetrahedral mesh. }
  \label{fig:cube3d_checkerboard}
\end{figure}

\begin{table}[!htbp]
\centering
\caption{3D checkerboard coefficient on the Delaunay tetrahedral meshes.  Panels report (a) PCG iterations,
(b) $N_{\mathrm{level}}/\mathrm{OC}$, and (c) $\delta_{\max}$.}
\label{tab:cube3d_slabs_pcg}
\begingroup
\small
\setlength{\tabcolsep}{2.5pt}
\tablepanel{a}{PCG iterations}

%% with Muelu
% \begin{tabular}{rrrrrrrrr}
% \toprule
% && \multicolumn{3}{c}{Two-level}
% & \multicolumn{4}{c}{Multilevel} \\
% \cmidrule(lr){3-5}\cmidrule(lr){6-9}
% $L$ & $\mu_{\min}$ & Ned & Ref & Alg & Ned & AMS & MueLu & ParELAG \\
% \midrule
% \csvreader[separator=semicolon,late after line={\\\paperafterrow}]
%   {numerics/paper_tables/slabs3d_iterations.csv}
%   {level=\papercolA,mu=\papercolB,geo_two=\papercolC,ref=\papercolD,alg=\papercolE,geo_multi=\papercolF,ams=\papercolG,muelu=\papercolH,parelag=\papercolI,afterrow=\paperafterrow}
%   {\papercolA & \papercolB & \papercolC & \papercolD & \papercolE & \papercolF & \papercolG & \papercolH & \papercolI}
% \bottomrule
% \end{tabular}

%% without Muelu
\begin{tabular}{rrrrrrrr}
\toprule
&& \multicolumn{3}{c}{Two-level}
& \multicolumn{3}{c}{Multilevel} \\
\cmidrule(lr){3-5}\cmidrule(lr){6-8}
$L$ & $\mu_{\min}$ & Ned & Ref & Alg & Ned & AMS & ParELAG \\
\midrule
\begin{filecontents*}{numerics/paper_tables/checkerboard3d_iterations.csv}
level;mu;geo_two;ref;alg;geo_multi;ams;muelu;parelag;afterrow
1;$1$;13;13;15;13;10;17;\textbf{7};
1;$10^{-2}$;25;\textbf{13};15;25;27;33;23;
1;$10^{-4}$;87;\textbf{14};17;87;102;130;142;
1;$10^{-6}$;107;\textbf{13};17;107;183;186;248;
2;$1$;12;12;15;13;13;17;\textbf{8};\midrule
2;$10^{-2}$;26;\textbf{14};16;28;29;33;24;
2;$10^{-4}$;114;\textbf{16};29;125;127;157;156;
2;$10^{-6}$;181;\textbf{16};27;175;282;276;303;
\end{filecontents*}
\csvreader[separator=semicolon,late after line={\\\paperafterrow}]
  {numerics/paper_tables/checkerboard3d_iterations.csv}
  {level=\papercolA,mu=\papercolB,geo_two=\papercolC,ref=\papercolD,alg=\papercolE,geo_multi=\papercolF,ams=\papercolG,parelag=\papercolI,afterrow=\paperafterrow}
  {\papercolA & \papercolB & \papercolC & \papercolD & \papercolE & \papercolF & \papercolG & \papercolI}
\bottomrule
\end{tabular}

\par\medskip
\setlength{\tabcolsep}{4pt}
\tablepanel{b}{$N_{\mathrm{level}}/\mathrm{OC}$}

%% with Muelu
% \begin{tabular}{rrrrrrrr}
% \toprule
% & & \multicolumn{3}{c}{Two-level}
% & \multicolumn{3}{c}{Multilevel} \\
% \cmidrule(lr){3-5}\cmidrule(lr){6-8}
% $L$ & $n$ & Ned & Ref & Alg & Ned & MueLu & ParELAG \\
% \midrule
% \csvreader[separator=semicolon,late after line={\\\paperafterrow}]
%   {numerics/paper_tables/slabs3d_complexity.csv}
%   {level=\papercolA,unknowns=\papercolB,geo_two=\papercolC,ref=\papercolD,alg=\papercolE,geo_multi=\papercolF,muelu=\papercolG,parelag=\papercolH,afterrow=\paperafterrow}
%   {\papercolA & \papercolB & \papercolC & \papercolD & \papercolE & \papercolF & \papercolG & \papercolH}
% \bottomrule
% \end{tabular}

%% without Muelu
\begin{tabular}{rrrrrrr}
\toprule
& & \multicolumn{3}{c}{Two-level}
& \multicolumn{2}{c}{Multilevel} \\
\cmidrule(lr){3-5}\cmidrule(lr){6-7}
$L$ & $n$ & Ned & Ref & Alg & Ned & ParELAG \\
\midrule
\begin{filecontents*}{numerics/paper_tables/checkerboard3d_complexity.csv}
level;unknowns;geo_two;ref;alg;geo_multi;muelu;parelag;afterrow
1;7,690;$2/1.086$;$2/6.446$;$2/3.924$;$2/1.086$;$2/1.223$;$2/1.137$--$1.161$;
2;67,804;$2/1.105$;$2/57.945$;$2/32.165$;$3/1.114$;$2/1.294$;$3/1.147$--$1.170$;
\end{filecontents*}
\csvreader[separator=semicolon,late after line={\\\paperafterrow}]
  {numerics/paper_tables/checkerboard3d_complexity.csv}
  {level=\papercolA,unknowns=\papercolB,geo_two=\papercolC,ref=\papercolD,alg=\papercolE,geo_multi=\papercolF,parelag=\papercolH,afterrow=\paperafterrow}
  {\papercolA & \papercolB & \papercolC & \papercolD & \papercolE & \papercolF & \papercolH}
\bottomrule
\end{tabular}

\par\medskip
\tablepanel{c}{$\delta_{\max}$}
\begin{tabular}{rrrr}
\toprule
$L$ & $\mu_{\min}$ & Ref & Alg \\
\midrule
\begin{filecontents*}{numerics/paper_tables/checkerboard3d_commuting.csv}
level;mu;ref;alg;afterrow
1;$1$;$2.02\times10^{-5}$;$2.70\times10^{-5}$;
1;$10^{-2}$;$7.91\times10^{-4}$;$8.69\times10^{-4}$;
1;$10^{-4}$;0.0586;0.0632;
1;$10^{-6}$;0.367;0.377;
2;$1$;$4.93\times10^{-6}$;$9.18\times10^{-6}$;\midrule
2;$10^{-2}$;$2.86\times10^{-4}$;$4.52\times10^{-4}$;
2;$10^{-4}$;0.0258;0.0376;
2;$10^{-6}$;0.311;0.333;
\end{filecontents*}
\csvreader[separator=semicolon,late after line={\\\paperafterrow}]
  {numerics/paper_tables/checkerboard3d_commuting.csv}
  {level=\papercolA,mu=\papercolB,ref=\papercolC,alg=\papercolD,afterrow=\paperafterrow}
  {\papercolA & \papercolB & \papercolC & \papercolD}
\bottomrule
\end{tabular}
\endgroup
\end{table}

Like the 2D cases, Table~\ref{tab:cube3d_slabs_pcg} shows that Ref and Alg retain an substantial
iteration advantage over Ned, particularly at high jump contrast. 
Ref is nearly insensitive to constrast and refinement, while Alg exhibits modest variation. 
Their similar performance supports the effectiveness of algebraic coarse variables on this mesh. 
On the other hand, AMS and ParELAG both have substantial deterioration in their iteration counts as jump contrast grows.

The main limitation is interpolation density. 
Ref and Alg have sharply growing operator complexities, so their lower iteration counts do not imply lower solve work. 
Alg reduces fill relative to Ref, but a practical 3D method still requires sparse harmonic extension. 
ParELAG's OC fluctuates because its coarse grid size changes with the coefficient contrast, but remain in a similar range. 
As we have seen in two dimensions, the commuting defects in 3D also do not consistently track
iteration counts and are not standalone convergence predictors.

To examine interpolation stability and smoothing quality separately, we numerically estimate $\widehat{\eta}_P$ and $\widehat{C}_{\mathrm{sm}}$ below. 
The smoothing diagnostic is restricted to $L=1$ because its dense implementation is impractical at $L=2$. 
A single double-precision matrix requires about $36.8$ GB.

\begin{table}[htbp]
\centering
\caption{Estimates of $\widehat{\eta}_P$ for the
three-dimensional checkerboard problem.}
\label{tab:cube3d_slabs_eta}
\begingroup
\small
\setlength{\tabcolsep}{5pt}
\begin{tabular}{rrrrrrr}
\toprule
& \multicolumn{3}{c}{$L=1$} & \multicolumn{3}{c}{$L=2$} \\
\cmidrule(lr){2-4}\cmidrule(lr){5-7}
$\mu_{\min}$ & Ned & Ref & Alg & Ned & Ref & Alg \\
\midrule
\begin{filecontents*}{numerics/paper_tables/checkerboard3d_stability.csv}
mu;geo_level1;ref_level1;alg_level1;geo_level2;ref_level2;alg_level2;afterrow
$1$;4.249;4.249;3.661;4.469;4.468;3.875;
$10^{-2}$;54.65;4.243;3.655;66.58;4.466;3.873;
$10^{-4}$;$5.178\times10^3$;3.952;3.355;$6.300\times10^3$;4.341;3.742;
$10^{-6}$;$1.395\times10^5$;9.516;9.578;$3.387\times10^5$;7.907;6.654;
\end{filecontents*}
\csvreader[separator=semicolon,late after line={\\\paperafterrow}]
  {numerics/paper_tables/checkerboard3d_stability.csv}
  {mu=\papercolA,geo_level1=\papercolB,ref_level1=\papercolC,alg_level1=\papercolD,geo_level2=\papercolE,ref_level2=\papercolF,alg_level2=\papercolG,afterrow=\paperafterrow}
  {\papercolA & \papercolB & \papercolC & \papercolD & \papercolE & \papercolF & \papercolG}
\bottomrule
\end{tabular}
\endgroup
\par\bigskip

\caption{Estimates of $\widehat{C}_{\mathrm{sm}}$ for the
three-dimensional checkerboard problem at $k=1$.}
\label{tab:cube3d_slabs_csm}
\begin{tabular}{rrrr}
\toprule
$\mu_{\min}$ & Ned & Ref & Alg \\
\midrule
\begin{filecontents*}{numerics/paper_tables/checkerboard3d_smoothing.csv}
mu;geo;ref;alg;afterrow
$1$;3.311;3.311;3.189;
$10^{-2}$;6.055;6.055;6.418;
$10^{-4}$;437.0;437.0;438.9;
$10^{-6}$;$4.105\times10^3$;$4.105\times10^3$;$4.002\times10^3$;
\end{filecontents*}
\csvreader[separator=semicolon,late after line={\\\paperafterrow}]
  {numerics/paper_tables/checkerboard3d_smoothing.csv}
  {mu=\papercolA,geo=\papercolB,ref=\papercolC,alg=\papercolD,afterrow=\paperafterrow}
  {\papercolA & \papercolB & \papercolC & \papercolD}
\bottomrule
\end{tabular}
\end{table}

Table~\ref{tab:cube3d_slabs_eta} shows that Ref and Alg retain $\widehat{\eta}P$ estimates below 10 across both meshes and all tested contrasts, whereas the estimates for Ned grow sharply with contrast.
Table~\ref{tab:cube3d_slabs_csm} instead shows strong contrast sensitivity
in the fine-space smoothing estimates $\widehat{C}_{\mathrm{sm}}$, which increases from approximately 3 to approximately  $4\times10^3$. 
Ned and Ref share this constant because they use the same smoother and fine space.
These results identify that the $OSS$-$\ell_1$ smoother is not independent of coefficient variation in $\mu$ in this experiment.
The growth of $\widehat{C}_{\mathrm{sm}}$ indicates increasingly unfavorable oscillatory modes for $OSS$-$\ell_1$ as the jump grows,  despite the low observed PCG iteration counts for Ref and Alg. A further detailed investigation of smoother quality is beyond the scope of this paper. 

\section{Conclusions}\label{sec:conc}
This work develops an \Hcurl AMG construction derived directly from the GAMG framework, 
offering a harmonic-extension-based alternative to classical topological and global energy-minimization approaches. 
Based on an interior/exterior splitting, we developed interpolation strategies in both a refinement-based and a fully algebraic setting. 
We established jump-independent interpolation stability on canonical finite-element spaces. 
2D numerical results demonstrate the robustness and practicality of our proposed interpolations.
Our methods achieves lower estimated solve work than N\'ed\'elec interpolation and the tested AMS configuration in the reported two-dimensional jump experiments. 
In three dimensions, the construction provides a proof of concept for coefficient-dependent interpolation, although its current density limits practical use.
The interior/exterior splitting provides a framework for developing sparse approximations that 
preserve the favorable interpolation stability observed in the numerical experiments.
Our future work will focus on developing sparse, practical three-dimensional interpolation within this framework.

\section*{Acknowledgements}
This work was performed under the auspices of the U.S. Department of Energy (DOE) by
Lawrence Livermore National Laboratory under Contract DE-AC52-07NA27344
(LLNL-JRNL-2025377). Support was provided in part by the DOE, Office of Science, Office of Advanced Scientific Computing
Research, and in part by the DOE National Nuclear Security Administration
(NNSA) under the Advanced Simulation and Computing (ASC) Program.

\bibliographystyle{siamplain}
\bibliography{references}

\end{document}